\documentclass[reqno,11pt]{article}
\usepackage{a4wide}
\usepackage{color,eucal,enumerate,mathrsfs}
\usepackage[normalem]{ulem}
\usepackage{mdwlist}
\usepackage{tikz-cd}
\usepackage{amsmath}
\usepackage{amssymb,epsfig,bbm}
\numberwithin{equation}{section}
\usepackage{hyperref}
\usepackage{hyphenat}
\usepackage{cite}
\usepackage[latin1]{inputenc}

\newcommand{\N}{\mathbb{N}}
\newcommand{\Q}{\mathbb{Q}}
\newcommand{\R}{\mathbb{R}}
\newcommand{\Z}{\mathbb{Z}}

\newcommand{\mm}{{\mbox{\boldmath$m$}}}

\newcommand{\sfd}{{\sf d}}
\newcommand{\sfe}{{\sf e}}
\newcommand{\sff}{{\sf f}}
\newcommand{\sfg}{{\sf g}}

\newcommand{\sfC}{{\sf C}}

\newcommand{\sfR}{{\sf R}}
\newcommand{\sfS}{{\sf S}}

\newcommand{\Kliminf}{K\kern-3pt-\kern-2pt\mathop{\rm lim\,inf}\limits}  
\renewcommand{\d}{{\mathrm d}}

\newcommand{\restr}[1]{\lower3pt\hbox{$|_{#1}$}}
\newcommand{\la}{{\langle}}                  
\newcommand{\ra}{{\rangle}}
\newcommand{\eps}{\varepsilon}  
\newcommand{\nchi}{{\raise.3ex\hbox{$\chi$}}}

\newcommand{\lims}{\varlimsup}

\newcommand{\fr}{\penalty-20\null\hfill$\blacksquare$}                      

\newcommand{\X}{{\rm X}}
\newcommand{\XX}{{\mathbb X}}
\newcommand{\YY}{{\mathbb Y}}
\newcommand{\T}{{\mathbb T}}
\newcommand{\nnorm}{\boldsymbol{\sf n}}

\renewcommand{\mm}{\mathfrak m}                                

\newenvironment{proof}{\removelastskip\par\medskip   
\noindent{\textit{Proof.}}\rm}{\penalty-20\null\hfill$\square$\par\medbreak}

\newtheorem{theorem}{Theorem}[section]

\newtheorem{corollary}[theorem]{Corollary}
\newtheorem{lemma}[theorem]{Lemma}
\newtheorem{proposition}[theorem]{Proposition}

\newtheorem{definition}[theorem]{Definition}

\newtheorem{remark}[theorem]{Remark}

\newcommand{\beq}{\begin{equation}}
\newcommand{\eeq}{\end{equation}}

\title{The Serre-Swan theorem for normed modules}
\author{Danka Lu\v{c}i\'{c}\thanks{SISSA, email: \textsf{dlucic@sissa.it}}
\ and Enrico Pasqualetto\thanks{SISSA, email: \textsf{epasqual@sissa.it}}}

\begin{document}
\maketitle
\begin{abstract}
The aim of this note is to analyse the structure of the $L^0$-normed
$L^0$-modules over a metric measure space. These are a tool
that has been introduced by N.\ Gigli to develop a differential calculus
on spaces verifying the Riemannian Curvature Dimension condition.
More precisely, we discuss under which conditions an $L^0$-normed $L^0$-module
can be viewed as the space of sections of a suitable measurable Banach bundle
and in which sense such correspondence
can be actually made into an equivalence of categories.
\bigskip

\end{abstract}
\tableofcontents
\section*{Introduction}
The last few years have seen an increasing interest in the study of the
class of metric measure spaces that satisfy the so-called ${\sf CD}(K,\infty)$ condition
\cite{Lott-Villani07,Sturm06I,Sturm06II}, which provides a synthetic notion of having Ricci
curvature bounded from below by some constant $K\in\R$. 
A reinforcement of such condition, which rules out the Finsler geometries,
has been introduced in \cite{AmbrosioGigliSavare11-2,Gigli12}, where the definition
of ${\sf RCD}(K,\infty)$
space appeared. We refer to the surveys \cite{Villani2016,Villani2017} for an overview
of this topic and a detailed list of references.
\medskip

An important contribution to the vast literature devoted to this subject is given
by the paper \cite{Gigli14}, in which the author proposed a notion of differential 
structure, precisely tailored for the family of $\sf RCD$ spaces. More in detail,
it is possible to develop a first-order calculus on any abstract metric measure space,
while a second-order one can be built only in presence of a lower Ricci curvature bound.
An object that plays a fundamental role in such construction is the concept
of $L^0$-normed $L^0$-module, which we are now going to describe.
\medskip

Let $(\X,\sfd,\mm)$ be any given metric measure space. We denote by $L^0(\mm)$
the space of all real-valued Borel functions on $\X$ (up to $\mm$-a.e.\ equality).
Then an $L^0(\mm)$-normed $L^0(\mm)$-module is an algebraic module $\mathscr M$
over the commutative ring $L^0(\mm)$, endowed with a pointwise norm
operator $|\cdot|:\,\mathscr M\to L^0(\mm)$ that is compatible with the module
structure, in a suitable sense (cf.\ Definition \ref{def:nm} below).
This terminology has been introduced by Gigli in \cite{Gigli14} and further
investigated in \cite{Gigli17}; it represents a variant of the concept of $L^\infty$-module,
due to Weaver \cite{Weaver99,Weaver01}, who was in turn inspired by the works of
Sauvageot \cite{Sauvageot89,Sauvageot90}. An analogue of the $L^0$-normed $L^0$-modules
is given by the `randomly normed spaces', for whose discussion we refer to \cite{HLR91}.
\medskip

A key example of $L^0$-normed $L^0$-module, which actually served as a motivation
for the introduction of this kind of structure, is the space $L^0(TM)$ of
measurable vector fields on a given Riemannian manifold $M$. An important observation
is that the module $L^0(TM)$ consists of the (measurable) sections of a suitable
vector bundle over $M$, more specifically of the tangent bundle $TM$.
With this remark in mind, a natural question arises:
\[\begin{array}{ll}
\text{Is it possible to view any `locally finitely-generated' }L^0\text{-normed}\\
L^0\text{-module as the space of sections of a suitable vector bundle?}
\end{array}\]
The purpose of the present paper is to address such problem. 
First of all, we propose a notion of measurable Banach bundle over $(\X,\sfd,\mm)$
having finite-dimensional fibers; cf.\ Definition \ref{def:MBB}. 
It turns out that the set of measurable sections of any such bundle inherits a natural
structure of $L^0(\mm)$-normed $L^0(\mm)$-module that is proper, meaning that it is
`locally finitely-generated' in the sense of Definition \ref{def:prop_mod}.
Then our main result (i.e.\ Theorem \ref{thm:Serre-Swan}) can be informally stated
in the following way:
\[\begin{array}{ll}
\text{The category of measurable Banach bundles over }(\X,\sfd,\mm)\text{ is}\\
\text{equivalent to that of proper }L^0(\mm)\text{-normed }L^0(\mm)\text{-modules.}
\end{array}\]
It is worth to study the class of proper modules since they actually occur in many
interesting situations -- for instance, the tangent module $L^0(T\X)$ associated to an
${\sf RCD}(K,N)$ space, described in the last paragraph of this introduction,
is always proper. We shall refer to our result as the `Serre-Swan theorem
for normed modules',
the reason being that its statement is fully analogous to that of a classical theorem,
whose two original formulations are due to Serre \cite{Serre55} and Swan \cite{Swan62}.
Among the several versions of this theorem one can find in the literature,
the one that is more similar in spirit to ours is that for smooth Riemannian manifolds
\cite{nestruev2003smooth}. Such result correlates the smooth vector bundles over 
a connected Riemannian manifold $M$ with the finitely-generated projective modules
over the ring $C^{\infty}(M)$. In this regard, some further details will be provided
in Appendix \ref{ap:comparison}, where we will also make a comparison with the present paper.
\medskip

We conclude this introduction by mentioning some special instances of our result
that already appeared in previous works. A structural characterisation of the Hilbert
modules, which are $L^0$-normed $L^0$-modules whose pointwise norm satisfies a
pointwise parallelogram identity, can be found in \cite[Theorem 1.4.11]{Gigli14}.
A refinement of such result for a specific module over finite-dimensional $\sf RCD$
spaces, which we now briefly describe, has been proved in \cite{GP16}.

When working with the differential structure of general metric measure spaces,
an essential role is played by the tangent module $L^0(T\X)$, whose construction
is based upon the well-established theory of the Sobolev spaces $W^{1,2}(\X,\sfd,\mm)$;
we refer to \cite[Definitions 2.2.1/2.3.1]{Gigli14} for its axiomatisation
(actually, the object considered therein is the $L^2(\mm)$-normed $L^\infty(\mm)$-module
$L^2(T\X)$, whose relation with the module $L^0(T\X)$ will be depicted in
Appendix \ref{ap:variant}). Nevertheless, the tangent module $L^0(T\X)$
might give no geometric information about the underlying space $\X$ (e.g.,
if there are no non-constant absolutely continuous curves in $\X$, then the module
$L^0(T\X)$ is trivial). This is not the case when we additionally assume a lower bound
on the Ricci curvature: Mondino and Naber showed in \cite{Mondino-Naber14}
that the rescalings around $\mm$-a.e.\ point $x$ of a finite-dimensional
$\sf RCD$ space $\X$ converge in the pointed-measured-Gromov-Hausdorff topology
to the Euclidean space of dimension $k(x)\leq N$. By `glueing together' these
Euclidean fibers, one obtains the Gromov-Hausdorff
tangent bundle $T_{\rm GH}\X$ (as done in \cite[Section 4]{GP16}),
which has -- a priori -- nothing to do with the
purely analytical tangent module $L^0(T\X)$. However, by relying upon some
rectifiability properties of the $\sf RCD$ spaces \cite{Mondino-Naber14,GP16-2,DPMR16,MK16},
it is possible to prove (cf.\ \cite[Theorem 5.1]{GP16}) that $L^0(T\X)$ is
isometrically isomorphic to the space of measurable sections of $T_{\rm GH}\X$.

\bigskip
\noindent{\bf Acknowledgment}

\noindent The authors would like to acknowledge Nicola Gigli for having pointed out
this problem and for his many invaluable suggestions.
\section{The language of normed modules}
\subsection{\texorpdfstring{$L^0$}{L0}-modules}
Let $\XX=(\X,\sfd,\mm)$ be a \emph{metric measure space}, which for our purposes means that
\begin{equation}\begin{split}
(\X,\sfd)&\quad\text{ is a complete and separable metric space,}\\
\mm\neq 0&\quad\text{ is a non-negative Radon measure on }\X.
\end{split}\end{equation}
The space $\XX$ will remain fixed for the whole paper.
\bigskip

We denote by $L^0(\mm)$ the space of (equivalence classes up to $\mm$-a.e.\ equality of) Borel functions from $\X$ to $\R$.
Such space is a topological vector space when endowed with the complete and separable distance $\sfd_{L^0(\mm)}$, defined by
\begin{equation}
\sfd_{L^0(\mm)}(f,g)\,:=\,\inf_{\delta>0}\Big[\delta+\mm\big(\big\{|f-g|>\delta\big\}\big)\Big]
\quad\text{ for every }f,g\in L^0(\mm),
\end{equation}
which metrizes the convergence in measure; see \cite{Bogachev07} for the details.
From an algebraic point of view, $L^0(\mm)$ has a natural structure of commutative topological ring,
whose multiplicative identity is given by (the equivalence class of) the function identically equal to $1$.

Given any Borel set $A\subseteq\X$, it holds that $(\nchi_{A})$, i.e.\ the ideal of
$L^0(\mm)$ generated by $\nchi_A$, can be naturally identified with $L^0(\mm\restr{A})$
and that the quotient ring $L^0(\mm)/(\nchi_{A})$
is isomorphic to $(\nchi_{\X\setminus A})$.
Moreover, if $(A_i)_{i=1}^n$ is a family of pairwise disjoint Borel subsets of $\X$, then
\begin{equation}\label{eq:direct_sum_vs_union}
(\nchi_A)\cong(\nchi_{A_1})\oplus\ldots\oplus(\nchi_{A_n}),\quad\text{ where we set }
A:=A_1\cup\ldots\cup A_n.
\end{equation}
In particular, it holds that $(\nchi_A)\oplus(\nchi_{\X\setminus A})\cong L^0(\mm)$ for every $A\subseteq\X$ Borel.
\bigskip

We now recall some general terminology about modules over commutative rings: given
any module $M$ over a commutative ring $R$, given a set $S\subseteq M$ and denoted by
$\Pi:\,\bigoplus_{v\in S}R\to M$ the map $(r_v)_{v\in S}\mapsto\sum_{v\in S}r_v\cdot v$,
we say that
\begin{itemize}
\item $S$ \emph{generates} $M$ provided the map $\Pi$ is surjective,
\item $S$ is \emph{independent} provided the map $\Pi$ is injective,
\item $S$ is a \emph{basis} of $M$ provided the map $\Pi$ is bijective.
\end{itemize}
An $R$-module $M$ is \emph{finitely-generated}
provided it is generated by a finite set $\{v_1,\ldots,v_n\}\subseteq M$.
Moreover, any $R$-module $M$ is said to be \emph{free} provided it admits a basis.

We also recall that $M$ is \emph{projective as $R$-module} if it satisfies the
following property: given two modules $N$, $P$ over $R$, a surjective module homomorphism
$f:\,N\to P$ and a module homomorphism $g:\,M\to P$, there exists a module homomorphism
$h:\,M\to N$ such that
\begin{equation}\begin{tikzcd}
& M \arrow[swap]{ld}{h} \arrow{d}{g}\\
N \arrow[swap]{r}{f} & P
\end{tikzcd}\end{equation}
is a commutative diagram. It turns out that $M$ is projective as $R$-module if and only if
there exists a module $Q$ over $R$ such that $M\oplus Q$ is a free $R$-module;
cf.\ \cite[Theorem 3.4]{hungerford2003algebra}.
\bigskip

Hereafter, we shall focus our attention on modules $M$ over the commutative ring $L^0(\mm)$.
We start by fixing some notation: given any Borel subset $A$ of $\X$, let us define
\begin{equation}
M\restr A:=\nchi_A\cdot M=\big\{\nchi_A\cdot v\;\big|\;v\in M\big\}.
\end{equation}
It turns out that $M\restr A$ can be viewed as a module over the ring $L^0(\mm\restr A)\sim(\nchi_A)$.
\begin{definition}
Let $M$ be an $L^0(\mm)$-module. Let $A\subseteq\X$ be a Borel set such that $\mm(A)>0$.
Then some elements $v_1,\ldots,v_n\in M$ form a \emph{local basis} on $A$ if
$\nchi_A\cdot v_1,\ldots,\nchi_A\cdot v_n\in M\restr A$ is a basis
for the $L^0(\mm\restr A)$-module $M\restr A$. In this case, we say that $M$ has \emph{dimension} $n$ on $A$.
Moreover, we say that $M$ has dimension $0$ on $A$ provided $M\restr A=\{0\}$.
\end{definition}

Notice that the previous notion of dimension is well-posed, because any two bases of $M\restr A$ must
have the same cardinality, as a consequence of the fact that $L^0(\mm\restr A)$
is a non-trivial commutative ring; see for instance \cite[Theorem 2.6]{COHN1966215}.
\begin{remark}{\rm
Any $L^0(\mm)$-module $M$ inherits a natural structure of $\R$-linear space,
as granted by the fact that the field $\R$ is (isomorphic to) a subring of $L^0(\mm)$.
\fr}\end{remark}
\begin{definition}[Dimensional decomposition]\label{def:dd}
Let $M$ be an $L^0(\mm)$-module. Then a Borel partition $(E_n)_{n\in\N\cup\{\infty\}}$
of $\X$ is said to be a \emph{dimensional decomposition} of $M$ provided
\begin{itemize}
\item[$\rm i)$] $M$ has dimension $n$ on $E_n$ for every $n\in\N$ with $\mm(E_n)>0$,
\item[$\rm ii)$] $M$ does not admit any finite basis on any Borel set $A\subseteq E_\infty$ with $\mm(A)>0$.
\end{itemize}
The dimensional decomposition, whenever it exists, is unique
up to $\mm$-a.e.\ equality: i.e.\ given any other sequence $(F_n)_{n\in\N\cup\{\infty\}}$
satisfying the same properties it holds that $\mm(E_n\Delta F_n)=0$ for every $n\in\N\cup\{\infty\}$.
\end{definition}
\begin{theorem}\label{thm:dd_fd}
Let $M$ be a finitely-generated $L^0(\mm)$-module. Then $M$ admits a
dimensional decomposition $E_0,\ldots,E_n$. 
\end{theorem}

The previous result is taken from \cite[Theorem 1.1]{GUO2011833}.
As a consequence, we have that:
\begin{proposition}\label{prop:fin-gen_are_proj}
Let $M$ be a finitely-generated $L^0(\mm)$-module. Then $M$ is projective as a
module over $L^0(\mm)$.
\end{proposition}
\begin{proof}
We know from Theorem \ref{thm:dd_fd} that $M$ admits a dimensional decomposition $E_0,\ldots,E_n$.
For every $k=1,\ldots,n$, let us choose a local basis $v^k_1,\ldots,v^k_k\in M\restr{E_k}$ for $M\restr{E_k}$.
Define $M'$ as the $L^0(\mm)$-module given by
\[M':=\bigoplus_{k=1}^n\underset{k\text{ times}}{\underbrace{(\nchi_{E_k})\oplus\ldots\oplus(\nchi_{E_k})}}.\]
Then we denote by $\Phi:\,M'\to M$ the following map:
\[\Phi\big((f^k_i)_{1\leq i\leq k\leq n}\big):=\sum_{k=1}^n\sum_{i=1}^k f^k_i\cdot v^k_i
\quad\text{ for every }(f^k_i)_{1\leq i\leq k\leq n}=(f^1_1,f^2_1,f^2_2,\ldots,f^n_1,\ldots,f^n_n)\in M'.\]
Hence $\Phi$ is a module isomorphism, so that it suffices to prove that $M'$ is projective. Call
\[Q:=\bigoplus_{k=1}^n\underset{k\text{ times}}{\underbrace{(\nchi_{\X\setminus E_k})
\oplus\ldots\oplus(\nchi_{\X\setminus E_k})}}.\]
It follows from \eqref{eq:direct_sum_vs_union} that $M'\oplus Q\cong\bigoplus_{k=1}^n L^0(\mm)^k$, which is
a free $L^0(\mm)$-module. Therefore one has that $M'$ (and accordingly also $M$) is projective, as required.
\end{proof}

We would also like to build a dimensional decomposition
on $L^0(\mm)$-modules that are not necessarily finitely-generated. To do so,
we need to endow such modules with some additional topological
structures. For this reason, we introduce the following definitions:
\begin{definition}[Locality/glueing]
Let $M$ be an $L^0(\mm)$-module. Then we say that
\begin{itemize}
\item[$\rm i)$] $M$ has the \emph{locality property} if for any $v\in M$ and any sequence $(A_n)_{n\in\N}$
of Borel subsets of $\X$ it holds that
\begin{equation}
\nchi_{A_n}\cdot v=0\;\;\;\text{for every }n\in\N
\quad\Longrightarrow\quad\nchi_{\bigcup_{n\in\N}A_n}\cdot v=0.
\end{equation}
\item[$\rm ii)$] $M$ has the \emph{glueing property} if for any $(v_n)_{n\in\N}\subseteq M$ and any
sequence $(A_n)_{n\in\N}$ of pairwise disjoint Borel subsets of $\X$ there exists an element
$v\in M$ such that
\begin{equation}
\nchi_{A_n}\cdot v=\nchi_{A_n}\cdot v_n\quad\text{ for every }n\in\N.
\end{equation}
\end{itemize}
\end{definition}

As one might expect, neither locality nor glueing are in general granted on (algebraic) modules
over the ring $L^0(\mm)$. Counterexamples can be easily built by suitably adapting
the arguments contained in Example 1.2.4 and Example 1.2.5, respectively, of \cite{Gigli14}.
\begin{remark}\label{rmk:dd_fg_mod}{\rm
It directly follows from Theorem \ref{thm:dd_fd} that any finitely-generated
$L^0(\mm)$-module has both the locality property and the glueing property.
\fr}\end{remark}
\begin{theorem}\label{thm:dd}
Let $M$ be an $L^0(\mm)$-module with the locality property and the glueing property.
Then $M$ admits a dimensional decomposition $(E_n)_{n\in\N\cup\{\infty\}}$.
\end{theorem}
\begin{proof}
Let $n\in\N$ be fixed. We define the family $\mathcal F_n$ as
\[\mathcal F_n:=\big\{A\subseteq\X\text{ Borel}\;\big|\;M\text{ has dimension }n\text{ on }A\big\}.\]
Then we denote by $E_n$ the $\mm$-essential union of the elements of $\mathcal F_n$. Since $\mathcal F_n$
is closed under taking subsets, we can write $E_n$ as $\bigcup_{i\in\N}A_i$ for some
sequence $(A_i)_{i\in\N}\subseteq\mathcal F_n$ of pairwise disjoint sets. For any $i\in\N$, choose
a local basis $v^i_1,\ldots,v^i_n\in M$ for $M\restr{A_i}$ on $A_i$. Hence take those elements $v_1,\ldots,v_n\in M\restr E_n$
satisfying $\nchi_{A_i}\cdot v_j=\nchi_{A_i}\cdot v^i_j$ for all $i\in\N$ and $j=1,\ldots,n$.
We now show that $v_1,\ldots,v_n$ is a local basis for $M$ on $E_n$:
\begin{itemize}
\item Let $v\in M\restr{E_n}$ be arbitrary. For every $i\in\N$, there exist functions
$f^i_1,\ldots,f^i_n\in L^0(\mm\restr{A_i})$ such that $\nchi_{A_i}\cdot v=\sum_{j=1}^n f^i_j\cdot v^i_j$.
Call $f_j:=\sum_{i\in\N}f^i_j\in L^0(\mm\restr{E_n})$ for any $j=1,\ldots,n$.
Since $\nchi_{A_i}\cdot\big(\sum_{j=1}^n f_j\cdot v_j\big)=\sum_{j=1}^n f^i_j\cdot v^i_j=\nchi_{A_i}\cdot v$
is satisfied for every $i\in\N$, it holds that $v=\sum_{j=1}^n f_j\cdot v_j$ by locality property,
proving that $v_1,\ldots,v_n$ generate $M$ on $E_n$.
\item Suppose that $\sum_{j=1}^n f_j\cdot v_j=0$ for some $f_1,\ldots,f_n\in L^0(\mm\restr{E_n})$.
In particular, we have that
$\sum_{j=1}^n(\nchi_{A_i}\,f_j)\cdot v^i_j=\nchi_{A_i}\cdot\big(\sum_{j=1}^n f_j\cdot v_j\big)=0$
for all $i\in\N$, whence $\nchi_{A_i}\,f_j=0$ holds for any $i\in\N$ and $j=1,\ldots,n$.
This grants that $f_1,\ldots,f_n=0$, so that $v_1,\ldots,v_n$ are independent on $E_n$.
\end{itemize}
Therefore $M$ has dimension $n$ on $E_n$. Now let us define $E_\infty:=\X\setminus\bigcup_{n\in\N}E_n$.
It only remains to show item ii). We argue by contradiction: suppose that $M$ is finitely-generated on
some Borel set $A\subseteq E_\infty$ of positive $\mm$-measure. Let $v_1,\ldots,v_n$ be
generators of $M\restr A$. Since $A\cap E_n=\emptyset$, the elements $v_1,\ldots,v_n$ cannot
form a basis for $M\restr A$, then there exist $i\in\{1,\ldots,n\}$ and $A_1\subseteq A$ Borel with
$\mm(A_1)>0$ such that $\nchi_{A_1}\cdot v_i$ can be written as an $L^0(\mm\restr{A_1})$-linear combination
of the $(\nchi_{A_1}\cdot v_j)$'s with $j\neq i$. Given that $A_1\cap E_{n-1}=\emptyset$,
we have that $\{v_j\,:\,j\neq i\}$ cannot be a basis for $M\restr{A_1}$, and so on.
By repeating the same argument finitely many times, we finally obtain a Borel set $A_n\subseteq A_1$
that is not $\mm$-negligible and that satisfies $M\restr{A_n}=\{0\}$. This cannot hold because $A\cap E_0=\emptyset$,
thus leading to a contradiction. This proves ii).
\end{proof}

In the sequel, we shall mainly focus on the following class of $L^0(\mm)$-modules,
which strictly contains all the finitely-generated $L^0(\mm)$-modules.
\begin{definition}[Proper modules]\label{def:prop_mod}
Let $M$ be an $L^0(\mm)$-module having the locality property and the glueing property.
Denote by $(E_n)_{n\in\N\cup\{\infty\}}$ its dimensional decomposition.
Then $M$ is said to be \emph{proper} provided $\mm(E_\infty)=0$.
\end{definition}

The following result shows that in the category of the $L^0(\mm)$-modules having both the
locality and the glueing property, any proper $L^0(\mm)$-module $M$ is projective. Nevertheless, any such module $M$ needs not be projective as $L^0(\mm)$-module.
\begin{proposition}
Let $M$ be a proper $L^0(\mm)$-module.
Consider any two $L^0(\mm)$-modules $N$, $P$ with the locality and the glueing property,
a surjective module homomorphism $f:\,N\to P$ and a module homomorphism $g:\,M\to P$.
Then there exists a module homomorphism $h:\,M\to N$ such that $f\circ h=g$.
\end{proposition}
\begin{proof}
Fix a dimensional decomposition $(E_n)_{n\in\N}$ of $M$. Given any $n\in\N\setminus\{0\}$,
choose a local basis $v^n_1,\ldots,v^n_n\in M\restr{E_n}$ for $M$ on $E_n$.
Since the map $f$ is surjective, for any $n\geq i\geq 1$ there exists $w^n_i\in N$
such that $f(w^n_i)=g(v^n_i)$. Now let $v\in M$ be fixed. Then there is a unique
family of functions $(\lambda^n_i)_{n\geq i\geq 1}\subseteq L^0(\mm)$ such that
each $\lambda^n_i$ is concentrated on $E_n$ and
\begin{equation}\label{eq:decomp_pr_mod}
\nchi_{E_n}\cdot v=\sum_{i=1}^n\lambda^n_i\cdot v^n_i\quad\text{ for every }n\in\N\setminus\{0\}.
\end{equation}
Since the module $N$ has the glueing property, we know that there exists an element $h(v)\in N$
such that $\nchi_{E_n}\cdot h(v)=\sum_{i=1}^n\lambda^n_i\cdot w^n_i$ for every $n\geq 1$.
Such element $h(v)$ is also uniquely determined because $N$ has the locality property,
therefore we defined a map $h:\,M\to N$. It follows from its very construction that
$h$ is a homomorphism of $L^0(\mm)$-modules, whence it only remains to prove
that $f\circ h=g$. To this aim, take any $v\in M$ and choose those $(\lambda^n_i)_{n\geq i\geq 1}$
that satisfy \eqref{eq:decomp_pr_mod}. Then for all $n\geq 1$ we have that
\[\begin{split}
\nchi_{E_n}\cdot(f\circ h)(v)&=f\big(\nchi_{E_n}\cdot h(v)\big)=
f\bigg(\sum_{i=1}^n\lambda^n_i\cdot w^n_i\bigg)=
\sum_{i=1}^n\lambda^n_i\cdot f(w^n_i)=
\sum_{i=1}^n\lambda^n_i\cdot g(v^n_i)\\
&=\nchi_{E_n}\cdot g(v),
\end{split}\]
whence accordingly $(f\circ h)(v)=g(v)$ by the locality property of $P$, as required.
\end{proof}
\subsection{\texorpdfstring{$L^0$}{L0}-normed \texorpdfstring{$L^0$}{L0}-modules}\label{subsec:L0_norm_L0_mod}
The notion of $L^0(\mm)$-normed $L^0(\mm)$-module has been introduced in \cite{Gigli14}
and then further investigated in \cite{Gigli17}, with the aim of building a
differential structure over $\XX$. We briefly recall here its definition, taken
from \cite[Definition 1.6]{Gigli17}:
\begin{definition}[$L^0$-normed $L^0$-module]\label{def:nm}
A topological $L^0(\mm)$-module $(\mathscr M,+,\,\cdot\,,\tau)$ is said to be \emph{$L^0(\mm)$-normed}
provided it is endowed with an operator $|\cdot|:\,\mathscr M\to L^0(\mm)$,
called \emph{pointwise norm}, which satisfies the following properties:
\begin{itemize}
\item[$\rm i)$] One has $|v|\geq 0$ and $|f\cdot v|=|f||v|$ in the $\mm$-a.e.\ sense
for every $v\in\mathscr M$ and $f\in L^0(\mm)$.
\item[$\rm ii)$] Given a Borel probability measure $\mm'$ on $\X$ with $\mm\ll\mm'\ll\mm$,
we have that
\begin{equation}\label{eq:dist_on_mod}
\sfd_{\mathscr M}(v,w):=\int|v-w|\wedge 1\,\d\mm'\quad\text{ for every }v,w\in\mathscr M
\end{equation}
is a complete distance on $\mathscr M$ that induces the topology $\tau$.
\end{itemize}
\end{definition}
The particular choice of the measure $\mm'$ could change the distance $\sfd_{\mathscr M}$,
but does not affect neither the completeness of $\sfd_{\mathscr M}$ nor its induced topology $\tau$.
\begin{remark}\label{rmk:operat_L0_cont}{\rm
It follows from the definition of $L^0(\mm)$-normed $L^0(\mm)$-module that the multiplication
by $L^0$-functions $\,\cdot:\,L^0(\mm)\times\mathscr M\to\mathscr M$ and the pointwise
norm $|\cdot|:\,\mathscr M\to L^0(\mm)$ are continuous operators.
\fr}\end{remark}

Fix any two $L^0(\mm)$-normed $L^0(\mm)$-modules $\mathscr M$ and $\mathscr N$.
A \emph{module morphism} $\Phi:\,\mathscr M\to\mathscr N$ is any
$L^0(\mm)$-linear operator such that $\big|\Phi(v)\big|\leq|v|$
holds $\mm$-a.e.\ for every $v\in\mathscr M$. The family
of all module morphisms from $\mathscr M$ to $\mathscr N$
will be denoted by ${\rm Mor}(\mathscr M,\mathscr N)$.
\begin{lemma}
Let $\mathscr M$ be an $L^0(\mm)$-normed $L^0(\mm)$-module. Then $\mathscr M$
has both the locality property and the glueing property.
\end{lemma}
\begin{proof}
\textsc{Locality.} Consider any $v\in\mathscr M$ and any sequence $(A_n)_n$ of Borel
subsets of $\X$ such that $\nchi_{A_n}\cdot v=0$ for every $n\in\N$. This means that
$\nchi_{A_n}|v|=|\nchi_{A_n}\cdot v|=0$ holds $\mm$-a.e.\ for any $n\in\N$. Let us
call $A:=\bigcup_n A_n$. Therefore $|\nchi_A\cdot v|=\nchi_A\,|v|\leq\sum_n\nchi_{A_n}|v|=0$
is satisfied $\mm$-a.e.\ in $\X$, showing that $\nchi_A\cdot v=0$. This grants that
$\mathscr M$ has the locality property.\\
\textsc{Glueing.} Let $(v_n)_n\subseteq\mathscr M$ and let $(A_n)_n$ be a sequence
of pairwise disjoint Borel sets in $\X$. Let us define $w_n:=\sum_{k=0}^n\nchi_{A_k}\cdot v_k$
for every $n\in\N$. Given that $\sum_{k=0}^\infty\mm'(A_k)\leq 1$, we know
that $\sum_{k\geq n}\mm'(A_k)\to 0$ as $n\to\infty$, whence accordingly
\[\lims_{n,m\to\infty}\sfd_{\mathscr M}(w_n,w_m)\overset{\eqref{eq:dist_on_mod}}=
\lims_{n,m\to\infty}\sum_{k=n\wedge m+1}^{n\vee m}\int_{A_k}|v_k|\wedge 1\,\d\mm'
\leq\lims_{n,m\to\infty}\sum_{k>n\wedge m}\mm'(A_k)=0,\]
which ensures that the sequence $(w_n)_n$ is $\sfd_{\mathscr M}$-Cauchy.
Call $v\in\mathscr M$ its limit. Given any $k\in\N$, it holds that
$\nchi_{A_k}\cdot w_n=\nchi_{A_k}\cdot v_k$ for all $n\geq k$, whence by passing
to the limit as $n\to\infty$ we get that $\nchi_{A_k}\cdot v=\nchi_{A_k}\cdot v_k$.
This shows that $\mathscr M$ satisfies the glueing property.
\end{proof}
\begin{remark}\label{rmk:prop_mod_separable}{\rm
Any proper $L^0(\mm)$-normed $L^0(\mm)$-module $\mathscr M$ is separable.

Indeed, call $(E_n)_{n\in\N}$ the dimensional decomposition of $\mathscr M$.
For any $n\in\N$, choose a local basis $v^n_1,\ldots,v^n_n\in\mathscr M\restr{E_n}$
for $\mathscr M\restr{E_n}$. Fix a countable dense subset $D$ of $L^0(\mm)$. Then the set
\[\bigg\{\sum_{n=1}^\infty\sum_{i=1}^n f^n_i\cdot v^n_i\;\bigg|\;
(f^n_i)_{1\leq i\leq n}\subseteq D\bigg\}\subseteq\mathscr M,\]
which is countable by construction, is dense in $\mathscr M$ by Remark
\ref{rmk:operat_L0_cont}.
\fr}\end{remark}
\begin{definition}[The category of proper $L^0$-normed $L^0$-modules]\label{def:NMod_pr}
The category having the proper $L^0(\mm)$-normed $L^0(\mm)$-modules as objects
and the module morphisms as arrows is denoted by $\mathbf{NMod}_{\rm pr}(\XX)$.
\end{definition}

A classical reference for the language of categories we shall make us of is given
by \cite{MacLane98}.
\section{The language of measurable Banach bundles}
\subsection{Measurable Banach bundles}
The aim of this section is to propose a notion of measurable Banach bundle, or briefly MBB,
over the given metric measure space ${\mathbb X}=(\X,\sfd,\mm)$.
An alternative definition of MBB, which does not perfectly fit
into our framework, can be found in \cite{GP16}.
\begin{definition}[MBB]\label{def:MBB}
We define a \emph{measurable Banach bundle} over the space $\XX$ as any quadruplet
$\overline{\T}=(T,\underline E,\pi,\nnorm)$, where
\begin{itemize}
\item[$\rm i)$] $\underline E=(E_n)_{n\in\N}$ is a Borel partition of $\X$,
\item[$\rm ii)$] the set $T:=\bigsqcup_{n\in\N}E_n\times\R^n$ is called \emph{total space}
and is always implicitly endowed with the $\sigma$-algebra
$\bigcap_{n\in\N}(\iota_n)_*\mathscr B(E_n\times\R^n)$, where
$\iota_n:\,E_n\times\R^n\hookrightarrow T$ denotes the inclusion map for every $n\in\N$,
\item[$\rm iii)$] the map sending any element $(x,v)\in T$ to its base point $x\in\X$ is denoted by $\pi:\,T\to\X$
and is called \emph{projection map},
\item[$\rm iv)$] $\nnorm:\,T\to[0,+\infty)$ is a measurable function with the property that
for any $n\in\N$ it holds that $\nnorm(x,\cdot)$ is a norm on $\R^n$
for $\mm$-a.e.\ point $x\in E_n$.
\end{itemize}
\end{definition}
Given $n\in\N$ and $x\in E_n$, we say that $(\overline{\T})_x:=\pi^{-1}\{x\}=\{x\}\times\R^n$ is
the \emph{fiber} of $\overline{\T}$ over $x$. We will often implicitly identify the fiber
$(\overline{\T})_x$ with the vector space $\R^n$ itself.
\begin{remark}\label{rmk:charact_sigma-algebra_on_T}{\rm
It is immediate to check that a subset $S$ of the total space $T$ of an MBB $\overline{\T}$
is measurable if and only if $S\cap(E_n\times\R^n)$ is a Borel subset of $E_n\times\R^n$ for any $n\in\N$.
\fr}\end{remark}

We now describe which are the (pre-)morphisms between any two given MBB's.
\begin{definition}[MBB pre-morphisms]
Let $\overline{\T}_1=(T_1,{\underline E}^1,\pi_1,\nnorm_1)$,
$\overline{\T}_2=(T_2,{\underline E}^2,\pi_2,\nnorm_2)$ be
MBB's over $\XX$. Then a measurable map $\overline\varphi:\,T_1\to T_2$ is said to
be an \emph{MBB pre-morphism} from $\overline{\T}_1$ to $\overline{\T}_2$ provided the diagram
\begin{equation}\begin{tikzcd}
T_1 \arrow{r}{\overline\varphi} \arrow[swap]{rd}{\pi_1} & T_2 \arrow{d}{\pi_2} \\
 & \X
\end{tikzcd}\end{equation}
is commutative and for $\mm$-a.e.\ $x\in\X$ it holds that
$\overline\varphi\restr{(\overline{\T}_1)_x}:\,\big((\overline{\T}_1)_x,\nnorm_1(x,\cdot)\big)\to\big((\overline{\T}_2)_x,\nnorm_2(x,\cdot)\big)$ is a linear $1$-Lipschitz map.
\end{definition}

We declare two MBB pre-morphisms $\overline\varphi,\overline\varphi':\,T_1\to T_2$ to be equivalent,
briefly $\overline\varphi\sim\overline\varphi'$, if
\begin{equation}\label{eq:equiv_pre-MBB_morphisms}
\overline\varphi\restr{(\overline{\T}_1)_x}=\overline\varphi'\restr{(\overline{\T}_1)_x}
\quad\text{ holds for }\mm\text{-a.e.\ }x\in\X.
\end{equation}
We are now finally in a position to define the category of measurable Banach bundles over $\XX$.
\begin{definition}[The category of MBB's]
The collection of measurable Banach bundles over $\XX$ and of
equivalence classes of MBB pre-morphisms form a category, which we shall denote by $\mathbf{MBB}(\XX)$.
\end{definition}
\subsection{The section functor}
Once a notion of measurable Banach bundle is given, it is natural to consider its
`measurable sections', namely those maps which assign (in a measurable way) to almost
every point of the underlying metric measure space an element of the fiber over such point.

It will turn out that the space $\Gamma(\T)$ of all measurable sections of a measurable Banach
bundle $\T$ is a proper $L^0$-normed $L^0$-module. The correspondence $\T\mapsto\Gamma(\T)$
can be made into a functor, called `section functor', from the category of measurable Banach
bundles to the category of proper $L^0$-normed $L^0$-modules.
\begin{definition}[Sections of an MBB]
Let $\overline\T=(T,\underline E,\pi,\nnorm)$ be (a representative of) an MBB over $\XX$.
Then we call \emph{(measurable) section} of $\overline\T$ any measurable right inverse of
the projection $\pi$, i.e.\ any measurable map $\overline s:\,\X\to T$ such that
$\pi\circ\overline s={\rm id}_\X$.
\end{definition}

Two given sections $\overline s_1,\overline s_2:\,\X\to T$ are equivalent provided
$\overline s_1(x)=\overline s_2(x)$ for $\mm$-a.e.\ $x\in\X$. The space of all
equivalence classes of sections of $\overline\T$ will be denoted by $\Gamma(\overline\T)$.
We add some structure to the set $\Gamma(\overline\T)$,
in order to get an $L^0(\mm)$-normed $L^0(\mm)$-module:
\begin{itemize}
\item[$\rm i)$]\textsc{Vector space.} Let $s_1,s_2\in\Gamma(\overline\T)$ and $\lambda\in\R$.
Pick a representative
$\overline s_i:\,\X\to T$ of $s_i$ for each $i=1,2$. Then we can pointwise define the
sections $\overline s_1+\overline s_2$ and $\lambda\,\overline s_1$ of $\overline\T$ as
\begin{equation}\begin{split}
(\overline s_1+\overline s_2)(x)&:=\overline s_1(x)+\overline s_2(x)\\
(\lambda\,\overline s_1)(x)&:=\lambda\,\overline s_1(x)
\end{split}\quad\text{ for every }x\in\X.
\end{equation}
Therefore we define $s_1+s_2\in\Gamma(\overline\T)$ and $\lambda\,s_1\in\Gamma(\overline\T)$ as the equivalence classes
of $\overline s_1+\overline s_2$ and $\lambda\,\overline s_1$, respectively. It can be readily
seen that these operations are well-defined and give to $\Gamma(\overline\T)$ a vector space structure.
\item[$\rm ii)$]\textsc{Multiplication by $L^0$-functions.} Fix $s\in\Gamma(\overline\T)$ and $f\in L^0(\mm)$.
Choose a representative $\overline s:\,\X\to T$ of
$s$ and a Borel version $\overline f:\,\X\to\R$ of $f$.
Then the map $\overline f\cdot\overline s:\,\X\to T$, which is given by
\begin{equation}
(\overline f\cdot\overline s)(x):=\overline f(x)\,\overline s(x)\in(\overline\T)_x
\quad\text{ for every }x\in\X,
\end{equation}
is a section of $\overline\T$. Hence we define $f\cdot s\in\Gamma(\overline\T)$ as the
equivalence class of $\overline f\cdot\overline s$. This yields a well-posed bilinear
operator $\,\cdot:\,L^0(\mm)\times\Gamma(\overline\T)\to\Gamma(\overline\T)$.
\item[$\rm iii)$]\textsc{Pointwise norm.} Consider a section $s\in\Gamma(\overline\T)$.
Pick a representative
$\overline s:\,\X\to T$ of $s$. Define the Borel function $|\overline s|:\,\X\to[0,+\infty)$ as
\begin{equation}\label{eq:def_ptwse_norm}
|\overline s|(x):=\nnorm\big(\overline s(x)\big)\quad\text{ for every }x\in\X.
\end{equation}
Then we denote by $|s|\in L^0(\mm)$ the equivalence class of the function $|\overline s|$.
This gives us a well-defined operator $|\cdot|:\,\Gamma(\overline\T)\to L^0(\mm)$.
\item[$\rm iv)$]\textsc{Topology on $\Gamma(\overline\T)$.} Pick a Borel probability measure
$\mm'$ on $\X$ such that $\mm\ll\mm'\ll\mm$. Then we define the distance $\sfd_{\Gamma(\overline\T)}$
on $\Gamma(\overline\T)$ as follows:
\begin{equation}\label{eq:def_distance_Gamma(T)}
\sfd_{\Gamma(\overline\T)}(s_1,s_2):=\int|s_1-s_2|\wedge 1\,\d\mm'\quad\text{ for every }s_1,s_2\in\Gamma(\overline\T).
\end{equation}
We denote by $\tau$ the topology induced by $\sfd_{\Gamma(\overline\T)}$.
\end{itemize}
It turns out that $\Gamma(\overline\T)$ is an $L^0(\mm)$-normed $L^0(\mm)$-module.
Furthermore, given any measurable Banach bundle $\T$ over the space $\XX$, we define
\begin{equation}
\Gamma(\T):=\Gamma(\overline\T)\quad\text{ for one (thus any) representative }\overline\T
\text{ of }\T.
\end{equation}
Well-posedness of such definition is granted by the fact that $\Gamma(\overline\T_1)$
and $\Gamma(\overline\T_2)$ are isomorphic as $L^0(\mm)$-normed $L^0(\mm)$-modules
whenever $\overline\T_1$ and $\overline\T_2$ are equivalent bundles.
\begin{remark}[Constant sections]\label{rmk:constant_sections}{\rm
In the forthcoming discussion, a key role will be played by those sections of $\T$
that are obtained in this way: for any $n\in\N$ and any vector $v\in\R^n$, we consider
the section $\boldsymbol v\in\Gamma(\T)$ that is identically equal to $v$ on $E_n$ and
null elsewhere.

More precisely, for any $n\in\N$ and any vector $v\in\R^n$, we define $\boldsymbol v\in\Gamma(\T)$
as the equivalence class of the section $\overline{\boldsymbol v}:\,\X\to T$, given by
\begin{equation}\label{eq:constant_sections_aux}
\overline{\boldsymbol v}(x):=\left\{\begin{array}{ll}
(x,v)\\
(x,0)
\end{array}\quad\begin{array}{ll}
\text{ if }x\in E_n,\\
\text{ if }x\in\X\setminus E_n,
\end{array}\right.
\end{equation}
where $\overline\T=(T,\underline E,\pi,\nnorm)$ is any chosen representative of $\T$.
\fr}\end{remark}
\begin{proposition}\label{prop:Gamma(T)_module}
The space $\Gamma(\T)$ is a proper $L^0(\mm)$-normed $L^0(\mm)$-module.
More precisely, for any representative $\overline\T=(T,\underline E,\pi,\nnorm)$ of the bundle $\T$
it holds that $\underline E=(E_n)_{n\in\N}$ constitutes a dimensional decomposition of $\Gamma(\T)$.
\end{proposition}
\begin{proof}
Fix $\overline\T=(T,\underline E,\pi,\nnorm)\in\T$ and $n\in\N$.
Denote by ${\sf e}_1,\ldots,{\sf e}_n$ the canonical basis of $\R^n$. Then consider
the sections $\boldsymbol{\sf e}_1,\ldots,\boldsymbol{\sf e}_n\in\Gamma(\T)$
defined in Remark \ref{rmk:constant_sections}. We claim that
\begin{equation}\label{eq:Gamma(T)_module_claim}
\boldsymbol{\sf e}_1,\ldots,\boldsymbol{\sf e}_n\quad\text{ is a local basis for }\Gamma(\T)\text{ on }E_n.
\end{equation}
Take any $s\in\Gamma(\T)$, with representative $\overline s:\,\X\to T$.
Since the map $\overline s\restr{E_n}:\,E_n\to E_n\times\R^n$ is Borel measurable by
Remark \ref{rmk:charact_sigma-algebra_on_T}, there exists a Borel function
$\overline c=(\overline c_1,\ldots,\overline c_n):\,E_n\to\R^n$ such that
$\overline s(x)=\big(x,\overline c(x)\big)$ holds for every $x\in E_n$.
Now extend each $\overline c_i$ to the whole $\X$ by declaring it equal to $0$ on
the complement of $E_n$. Hence
$\nchi_{E_n}\cdot\overline s=\sum_{i=1}^n\overline c_i\cdot\overline{\boldsymbol{\sf e}}_i$,
where $\overline{\boldsymbol{\sf e}}_1,\ldots,\overline{\boldsymbol{\sf e}}_n$ are defined
as in \eqref{eq:constant_sections_aux}. Calling $c_i\in L^0(\mm)$ the equivalence class
of $\overline c_i$ for every $i=1,\ldots,n$, we deduce that
$\nchi_{E_n}\cdot s=\sum_{i=1}^n c_i\cdot\boldsymbol{\sf e}_i$,
which grants that $\boldsymbol{\sf e}_1,\ldots,\boldsymbol{\sf e}_n$ generate $\Gamma(\T)$ on $E_n$.

Now suppose that $\sum_{i=1}^n c_i\cdot\boldsymbol{\sf e}_i=0$ for some
$c_1,\ldots,c_n\in L^0(\mm)$. Choose a Borel representative $\overline c_i:\,\X\to\R$ of
each $c_i$, whence $\big(\overline c_1(x),\ldots,\overline c_n(x)\big)=
\big(\sum_{i=1}^n\overline c_i\cdot\overline{\boldsymbol{\sf e}}_i\big)(x)=0$
holds for $\mm$-a.e.\ $x\in E_n$, in other words $\nchi_{E_n}c_1,\ldots,\nchi_{E_n}c_n=0$.
Therefore the sections $\boldsymbol{\sf e}_1,\ldots,\boldsymbol{\sf e}_n$ are independent on $E_n$.
This yields \eqref{eq:Gamma(T)_module_claim} and accordingly the statement.
\end{proof}

In order to define the functor $\Gamma$ from $\mathbf{MBB}(\XX)$ to $\mathbf{NMod}_{\rm pr}(\XX)$,
it only remains to declare how it behaves on morphisms, namely to associate to any
MBB morphism $\varphi\in{\rm Mor}(\T_1,\T_2)$ a suitable morphism
$\Gamma(\varphi):\,\Gamma(\T_1)\to\Gamma(\T_2)$ of proper $L^0(\mm)$-normed $L^0(\mm)$-modules.
\bigskip

Let $\overline\T_i=(T_i,\underline E^i,\pi_i,\nnorm_i)$ be representatives of MBB's over $\XX$ for $i=1,2$.
Take a section $\overline s$ of $\overline\T_1$ and a pre-morphism
$\overline\varphi:\,T_1\to T_2$. Since $\overline\varphi\circ\overline s:\,\X\to T_2$ is measurable
as composition of measurable maps and $\pi_2\circ\overline\varphi\circ\overline s=\pi_1\circ\overline s={\rm id}_\X$,
we conclude that $\overline\varphi\circ\overline s$ is a section of $\overline\T_2$.

Now let us call $\T_1$, $\T_2$ and $\varphi$ the equivalence classes of $\overline\T_1$,
$\overline\T_2$ and $\overline\varphi$, respectively. Then we define
$\Gamma(\varphi):\,\Gamma(\T_1)\to\Gamma(\T_2)$ as follows: given any $s\in\Gamma(\T_1)$,
we set
\begin{equation}
\Gamma(\varphi)(s):=\text{the equivalence class of }\overline\varphi\circ\overline s,
\quad\text{ where }\overline s\text{ is any representative of }s.
\end{equation}
In the next result, we shall prove that $\Gamma(\varphi)$ is actually a module morphism.
\begin{lemma}\label{lemma:Gamma_functor}
Let $\T_1$, $\T_2$ be two measurable Banach bundles over $\XX$ and let
$\varphi\in{\rm Mor}(\T_1,\T_2)$. Then $\Gamma(\varphi)\in{\rm Mor}\big(\Gamma(\T_1),\Gamma(\T_2)\big)$.
\end{lemma}
\begin{proof}
It suffices to show that for any $s_1,s_2\in\Gamma(\T_1)$ and $f_1,f_2\in L^0(\mm)$ one has
\begin{equation}\label{eq:Gamma(phi)_morphism_aux}\begin{split}
\Gamma(\varphi)(f_1\cdot s_1+f_2\cdot s_2)&
=f_1\cdot\Gamma(\varphi)(s_1)+f_2\cdot\Gamma(\varphi)(s_2),\\
\big|\Gamma(\varphi)(s_1)\big|&\leq|s_1|\quad\mm\text{-a.e.\ in }\X.
\end{split}\end{equation}
Choose representatives $\overline\T_i=(T_i,\underline E^i,\pi_i,\nnorm_i)$ of $\T_i$,
$\overline\varphi:\,T_1\to T_2$ of $\varphi$ and $\overline s_i:\,\X\to T_1$ of $s_i$
for each $i=1,2$. Further, choose Borel functions $\overline f_1,\overline f_2:\,\X\to\R$
that are representatives of $f_1$ and $f_2$, respectively. Hence for $\mm$-a.e.\ point $x\in\X$
it holds that
\[\begin{split}
\big(\overline\varphi\circ(\,\overline f_1\cdot\overline s_1+
\overline f_2\cdot\overline s_2)\big)(x)&=
\overline\varphi\big(\,\overline f_1(x)\,\overline s_1(x)+
\overline f_2(x)\,\overline s_2(x)\big)\\
&=\overline f_1(x)\,(\overline\varphi\circ\overline s_1)(x)+
\overline f_2(x)\,(\overline\varphi\circ\overline s_2)(x),
\end{split}\]
whence $\Gamma(\varphi)(f_1\cdot s_1+f_2\cdot s_2)=f_1\cdot\Gamma(\varphi)(s_1)
+f_2\cdot\Gamma(\varphi)(s_2)$, i.e.\ the first in \eqref{eq:Gamma(phi)_morphism_aux}.

To prove the second one, observe that for $\mm$-a.e.\ $x\in\X$ one has that
\[
|\overline\varphi\circ\overline s_1|(x)=\nnorm_2\big((\overline\varphi\circ\overline s_1)(x)\big)
=(\nnorm_2\circ\overline\varphi)\big(\overline s_1(x)\big)\leq\nnorm_1\big(\overline s_1(x)\big)=|\overline s_1|(x),
\]
so that $\big|\Gamma(\varphi)(s_1)\big|\leq|s_1|$ holds $\mm$-a.e.\ in $\X$. Therefore the thesis is achieved.
\end{proof}
\begin{definition}[Section functor]\label{def:section_functor}
The covariant functor $\Gamma:\,\mathbf{MBB}(\XX)\to\mathbf{NMod}_{\rm pr}(\XX)$, which
associates to any object $\T$ of $\mathbf{MBB}(\XX)$ the object $\Gamma(\T)$ of
$\mathbf{NMod}_{\rm pr}(\XX)$ and to any morphism $\varphi:\,\T_1\to\T_2$ the morphism
$\Gamma(\varphi):\,\Gamma(\T_1)\to\Gamma(\T_2)$, is called \emph{section functor} on $\XX$.
\end{definition}
\section{Main result: the Serre-Swan theorem}
Our main result states that the section functor is actually an equivalence of categories.
We shall refer to such result as the \emph{Serre-Swan theorem for normed modules}.
First, we prove a technical lemma that provides us with a suitable
dense subset of the space of all measurable sections of a measurable Banach bundle.
Then such density result (Lemma \ref{lemma:density_simple_sections}) will be needed
to show that the section functor is `essentially surjective'
(Proposition \ref{prop:Gamma_essentially_surj}) and fully faithful
(Proposition \ref{prop:Gamma_full_faithful}). Finally, the Serre-Swan theorem
(Theorem \ref{thm:Serre-Swan}) will immediately follow.
\bigskip

Given a measurable Banach bundle $\T$ over $\XX$ and any $n\in\N$, we set
\begin{equation}
\sfS(\T,n):=\left\{\sum_{i\in\N}\nchi_{A_i}\cdot\boldsymbol q^i\;\bigg|\;
(A_i)_{i\in\N}\text{ is a Borel partition of }E_n\text{, }(q^i)_{i\in\N}\subseteq\Q^n\right\},
\end{equation}
where the `constant sections' $\boldsymbol q^i\in\Gamma(\T)$ are defined as in Remark
\ref{rmk:constant_sections}. Note that any element of the form
$\sum_{i\in\N}\nchi_{A_i}\cdot\boldsymbol q^i\in\Gamma(\T)$ is well-defined since the sets
$A_i$'s are pairwise disjoint.

Then we define the family $\sfS(\T)\subseteq\Gamma(\T)$ of `simple sections' of $\T$ as follows:
\begin{equation}
\sfS(\T):=\big\{t\in\Gamma(\T)\;\big|\;
\nchi_{E_n}\cdot t\in\sfS(\T,n)\text{ for every }n\in\N\big\}.
\end{equation}
We now show that such class of sections, which is a $\Q$-vector space, is actually dense in $\Gamma(\T)$:
\begin{lemma}\label{lemma:density_simple_sections}
Let $\T$ be a measurable Banach bundle over $\XX$. Then $\sfS(\T)$ is dense in $\Gamma(\T)$.
\end{lemma}
\begin{proof}
Let $s\in\Gamma(\T)$ and $\eps>0$ be fixed. Choose any Borel probability measure $\mm'$
on $\X$ such that $\mm\ll\mm'\ll\mm$ and define the distance $\sfd_{\Gamma(\T)}$ on $\Gamma(\T)$
as in \eqref{eq:def_distance_Gamma(T)}. We aim to find a simple section $t\in\sfS(\T)$
that satisfies $\sfd_{\Gamma(\T)}(s,t)\leq\eps$. To do so, choose representatives
$\overline\T=(T,\underline E,\pi,\nnorm)$ and $\overline s:\,\X\to T$ of $\T$ and $s$,
respectively. We can clearly suppose without loss of generality that $\nnorm(x,\cdot)$ is
a norm for every $x\in\X$. Given any $n\in\N$, let us define
\[E_{n,k}:=\left\{x\in E_n\;\bigg|
\;k-1<\sup_{q\in\Q^n\setminus\{0\}}\frac{\nnorm(x,q)}{|q|}\leq k\right\}
\quad\text{ for every }k\in\N.\]
Since $E_n\ni x\mapsto\nnorm(x,q)/|q|$ is Borel for every $q\in\Q^n\setminus\{0\}$,
we know that each $E_{n,k}$ is Borel. Moreover, the fact that any two norms on $\R^n$ are
equivalent grants that the supremum in the definition of $E_{n,k}$ is finite for
every $x\in E_n$, whence for all $n\in\N$ we have that $(E_{n,k})_{k\in\N}$ constitutes
a Borel partition of $E_n$.
For any $n,k\in\N$, call $\overline s_{n,k}:\,E_{n,k}\to\R^n$ that Borel map for which
$\overline s(x)=\big(x,\overline s_{n,k}(x)\big)$ for every $x\in E_{n,k}$.
It is well-known that there exists a Borel map $\overline t_{n,k}:\,E_{n,k}\to\R^n$ whose
image is a finite subset of $\Q^n$ and satisfying
\begin{equation}\label{eq:density_simple_sections_aux}
\int_{E_{n,k}}\big|k\,\overline s_{n,k}(x)-k\,\overline t_{n,k}(x)\big|\wedge 1\,\d\mm'(x)
\leq\frac{\eps}{2^{n+k}}.
\end{equation}
Given that $\nnorm(x,c)\leq k\,|c|$ holds for every $x\in E_{n,k}$ and $c\in\R^n$,
we deduce from \eqref{eq:density_simple_sections_aux} that
\begin{equation}\label{eq:density_simple_sections_aux2}
\int_{E_{n,k}}\nnorm\big(x,\overline s_{n,k}(x)-\overline t_{n,k}(x)\big)\wedge 1\,\d\mm'(x)
\leq\frac{\eps}{2^{n+k}}.
\end{equation}
Now let us denote by $\overline t:\,\X\to T$ the measurable map such that
$\overline t\restr{E_{n,k}}=({\rm id}_{E_{n,k}},\overline t_{n,k})$ holds for every $n,k\in\N$,
which is meaningful since $(E_{n,k})_{n,k\in\N}$ is a partition of $\X$.
Call $t\in\Gamma(\T)$ the equivalence class of $\overline t$. Notice that
$t\in\sfS(\T)$ by construction. Property \eqref{eq:density_simple_sections_aux2} yields
\[\sfd_{\Gamma(\T)}(s,t)=\int|s-t|\wedge 1\,\d\mm'
=\sum_{n,k\in\N}\int_{E_{n,k}}\nnorm\big(x,\overline s_{n,k}(x)-\overline t_{n,k}(x)\big)\wedge 1\,\d\mm'(x)
\leq\sum_{n,k\in\N}\frac{\eps}{2^{n+k}}=\eps,\]
which gives the thesis.
\end{proof}
\begin{proposition}\label{prop:Gamma_essentially_surj}
Let $\mathscr M$ be a proper $L^0(\mm)$-normed $L^0(\mm)$-module.
Then there exists a measurable Banach bundle $\T$ over $\XX$ such
that $\Gamma(\T)$ is isomorphic to $\mathscr M$.
\end{proposition}
\begin{proof}
Let $(E_n)_{n\in\N}$ be a dimensional decomposition of the module $\mathscr M$. 
Set $\underline E:=(E_n)_{n\in\N}$ and take $T$, $\pi$ as in the definition of MBB.
In order to define $\nnorm$, fix a sequence $(v_n)_{n\in\N}\subseteq\mathscr M$ such
that the elements $v_1,\ldots,v_n$ form a local basis for $\mathscr M$ on $E_n$ for each $n\in\N$.
Fix $n\in\N$. We define the linear and continuous operator $P_n:\,\R^n\to\mathscr M$
in the following way:
\[P_n(c):=\nchi_{E_n}\cdot(c_1\,v_1+\ldots+c_n\,v_n)\in\mathscr M
\quad\text{ for every }c=(c_1,\ldots,c_n)\in\R^n.\]
For any $q\in\Q^n$, choose any Borel representative $\overline{\big|P_n(q)\big|}:\,\X\to[0,+\infty)$
of $\big|P_n(q)\big|\in L^0(\mm)$. Hence there is a Borel set $N_n\subseteq E_n$,
with $\mm(N_n)=0$, such that for every $x\in E_n\setminus N_n$ it holds
\begin{equation}\label{eq:Gamma_essentially_surj_aux}\begin{split}
\overline{\big|P_n(q^1)+P_n(q^2)\big|}(x)\leq
\overline{\big|P_n(q^1)\big|}(x)+\overline{\big|P_n(q^2)\big|}(x)&\quad\text{ for every }q^1,q^2\in\Q^n,\\
\overline{\big|P_n(\lambda\,q)\big|}(x)=|\lambda|\,\overline{\big|P_n(q)\big|}(x)
&\quad\text{ for every }\lambda\in\Q\text{ and }q\in\Q^n,\\
\overline{\big|P_n(q)\big|}(x)>0&\quad\text{ for every }q\in\Q^n\setminus\{0\}.
\end{split}\end{equation}
Then let us define
\begin{equation}\label{eq:def_nnorm}
\nnorm(x,q):=\overline{\big|P_n(q)\big|}(x)
\quad\text{ for every }x\in E_n\setminus N_n\text{ and }q\in\Q^n.
\end{equation}
We deduce from \eqref{eq:Gamma_essentially_surj_aux} that $\nnorm(x,\cdot)$ is
a norm on $\Q^n$ for every $x\in E_n\setminus N_n$. In particular it is uniformly
continuous, whence it can be uniquely extended to a uniformly continuous map on
the whole $\R^n$, still denoted by $\nnorm(x,\cdot)$. By approximation, we see that
such extension is actually a norm on $\R^n$.
Finally, we set $\nnorm(x,c):=0$ for every $x\in N_n$ and $c\in\R^n$. We thus built a
function $\nnorm:\,T\to[0,+\infty)$. We claim that
\begin{equation}\label{eq:Gamma_essentially_surj_claim}
\nnorm\restr{E_n\times\R^n}\text{ is a Carath\'{e}odory function}
\quad\text{ for every }n\in\N,
\end{equation}
which grants that each $\nnorm\restr{E_n\times\R^n}$ is Borel, so accordingly that
$\nnorm$ is measurable by Remark \ref{rmk:charact_sigma-algebra_on_T}. First of all,
fix $n\in\N$ and notice that the function $\nnorm(x,\cdot):\,\R^n\to[0,+\infty)$ is
continuous for every $x\in E_n$. Moreover, given any $c\in\R^n$ and a
sequence $(q^k)_{k\in\N}\subseteq\Q^n$ converging to $c$, we have that
$\nnorm(x,c)=\lim_k\nnorm(x,q^k)=\lim_k\overline{\big|P_n(q^k)\big|}(x)$
for every $x\in E_n\setminus N_n$, whence the function $\nnorm(\cdot,c):\,E_n\to[0,+\infty)$ is Borel
as pointwise limit of a sequence of Borel functions. Therefore the claim
\eqref{eq:Gamma_essentially_surj_claim} is proved. We thus deduce that
$\overline\T:=(T,\underline E,\pi,\nnorm)$ is an MBB over the space $\XX$.
Then let us denote by $\T$ the equivalence class of $\overline\T$.

In order to get the thesis, we want to exhibit a module
isomorphism $I:\,\Gamma(\T)\to\mathscr M$, namely an $L^0(\mm)$-linear
map preserving the pointwise norm. We proceed as follows:
given any $s\in\Gamma(\T)$, choose a representative $\overline s:\,\X\to T$.
For any $n\in\N$, pick $\overline c^n:\,\X\to\R^n$ Borel such that
$\overline s(x)=\big(x,\overline c^n(x)\big)$ for every $x\in E_n$ and
call $c^n_1,\ldots,c^n_n\in L^0(\mm)$ those elements for which
$(c^n_1,\ldots,c^n_n)$ is the equivalence class of $\overline c^n$.
Now let us define
\begin{equation}
I(s):=\sum_{n\in\N}\nchi_{E_n}\cdot(c^n_1\cdot v_1+\ldots+c^n_n\cdot v_n)\in\mathscr M.
\end{equation}
One can easily see that the resulting map $I:\,\Gamma(\T)\to\mathscr M$ is a
(well-defined) $L^0(\mm)$-linear and continuous operator. We show that it is surjective:
fix any $v\in\mathscr M$, whence for each $n\in\N$ there exist $c^n_1,\ldots,c^n_n\in L^0(\mm)$
such that $\nchi_{E_n}\cdot v=\nchi_{E_n}\cdot(c^n_1\cdot v_1+\ldots+c^n_n\cdot v_n)$.
Pick any Borel representative $\overline c^n:\,\X\to\R^n$ of $(c^n_1,\ldots,c^n_n)$
and define $\overline s:\,\X\to T$ as $\overline s(x):=\big(x,\overline c^n(x)\big)$
for every $n\in\N$ and $x\in E_n$. Hence the equivalence class $s\in\Gamma(\T)$ of
$\overline s$ satisfies $I(s)=v$, thus proving that the map $I$ is surjective.
It only remains to prove that $\big|I(s)\big|=|s|$ holds $\mm$-a.e.\ in $\X$ for every
$s\in\Gamma(\T)$. First of all, for any $n\in\N$ and $q\in\Q^n$ one has that
$I(\boldsymbol q)=P_n(q)$, where the definition of $\boldsymbol q$ is taken from
Remark \ref{rmk:constant_sections}. Therefore
\begin{equation}\label{eq:Gamma_essentially_surj_aux2}
\big|I(\boldsymbol q)\big|=\big|P_n(q)\big|\overset{\eqref{eq:def_nnorm}}=
\nnorm\circ\boldsymbol q\overset{\eqref{eq:def_ptwse_norm}}=|\boldsymbol q|
\quad\text{ holds }\mm\text{-a.e.\ in }\X.
\end{equation}
We then directly deduce from \eqref{eq:Gamma_essentially_surj_aux2} and
the $L^0(\mm)$-linearity of $I$ that the equality $\big|I(t)\big|=|t|$ is
verified $\mm$-a.e.\ for every simple section $t\in\sfS(\T)$.
Recall that $\sfS(\T)$ is dense in $\Gamma(\T)$, as seen in
Lemma \ref{lemma:density_simple_sections}. Since both $I$ and the pointwise
norm are continuous operators, we finally conclude that $\big|I(s)\big|=|s|$
holds $\mm$-a.e.\ for every $s\in\Gamma(\T)$. Therefore $I$ preserves the
pointwise norm, thus completing the proof.
\end{proof}
\begin{proposition}\label{prop:Gamma_full_faithful}
The section functor $\Gamma:\,\mathbf{MBB}(\XX)\to\mathbf{NMod}_{\rm pr}(\XX)$
is full and faithful.
\end{proposition}
\begin{proof}
\textsc{Faithful.} Fix two measurable Banach bundles $\T_1$, $\T_2$ and
two different bundle morphisms $\varphi,\psi\in{\rm Mor}(\T_1,\T_2)$.
Choose a representative $\overline T_i=(T_i,\underline E^i,\pi_i,\nnorm_i)$
of $\T_i$ for $i=1,2$, then representatives $\overline\varphi,\overline\psi:\,T_1\to T_2$
of $\varphi$ and $\psi$, respectively. Hence there exist $n\in\N$ and
a Borel set $E\subseteq E^1_n$, with $\mm(E)>0$, such that
$\overline\varphi\restr{(\overline\T_1)_x}\neq\overline\psi\restr{(\overline\T_1)_x}$
for every $x\in E$. Let us denote by $\sfe_1,\ldots,\sfe_n$ the canonical basis of $\R^n$.
Therefore there exists $k\in\{1,\ldots,n\}$ such that
\[\mm\big(\big\{x\in E\,:\,\overline\varphi(x,\sfe_k)\neq\overline\psi(x,\sfe_k)\big\}\big)>0.\]
This means that
$\overline\varphi\circ\overline{\boldsymbol{{\sfe}}}_k$ is not $\mm$-a.e.\ coincident with
$\overline\psi\circ\overline{\boldsymbol{{\sfe}}}_k$, where the section $\overline{\boldsymbol{{\sfe}}}_k$
is defined as in Remark \ref{rmk:constant_sections}, whence
$\Gamma(\varphi)(\boldsymbol{\sfe}_k)\neq\Gamma(\psi)(\boldsymbol{\sfe}_k)$.
This implies that $\Gamma(\varphi)\neq\Gamma(\psi)$, thus proving that the functor
$\Gamma$ is faithful.\\
\textsc{Full.} Fix measurable Banach bundles $\T_1$, $\T_2$ and
a module morphism $\Phi:\,\Gamma(\T_1)\to\Gamma(\T_2)$. We aim to show
that there exists a bundle morphism $\varphi\in{\rm Mor}(\T_1,\T_2)$
such that $\Phi=\Gamma(\varphi)$. Since the ideas of the proof are similar in spirit
to those that have been used for proving Proposition \ref{prop:Gamma_essentially_surj},
we will omit some details. Choose a representative $\overline\T_i=(T_i,\underline E^i,\pi_i,\nnorm_i)$
of $\T_i$ for $i=1,2$. We define the Borel sets $F_{n,m}\subseteq\X$ as
\[F_{n,m}:=E^1_n\cap E^2_m\quad\text{ for every }n,m\in\N.\]
For any $n\in\N$ and $q\in\Q^n$, consider the section $\boldsymbol q\in\Gamma(\T_1)$ as
in Remark \ref{rmk:constant_sections} and
choose a representative $\overline{\Phi(\boldsymbol q)}:\,\X\to T_2$
of $\Phi(\boldsymbol q)\in\Gamma(\T_2)$. Given any $n,m\in\N$, there exists a Borel
subset $N_{n,m}$ of $F_{n,m}$, with $\mm(N_{n,m})=0$, such that for every
$x\in F_{n,m}\setminus N_{n,m}$ it holds
\begin{equation}\label{eq:Gamma_full_aux}\begin{split}
\overline{\Phi({\boldsymbol q}^1+{\boldsymbol q}^2)}(x)=
\overline{\Phi({\boldsymbol q}^1)}(x)+\overline{\Phi({\boldsymbol q}^2)}(x)
&\quad\text{ for every }q^1,q^2\in\Q^n,\\
\overline{\Phi(\lambda\,\boldsymbol q)}(x)=\lambda\,\overline{\Phi(\boldsymbol q)}(x)
&\quad\text{ for every }\lambda\in\Q\text{ and }q\in\Q^n,\\
\nnorm_2\big(\,\overline{\Phi(\boldsymbol q)}(x)\big)\leq\nnorm_1\big(\overline{\boldsymbol q}(x)\big)
&\quad\text{ for every }q\in\Q^n.
\end{split}\end{equation}
Then let us define
\begin{equation}
\overline\varphi(x,q):=\left\{\begin{array}{ll}
\overline{\Phi(\boldsymbol q)}(x)\\
0_{\R^m}
\end{array}\quad\begin{array}{ll}
\text{ for every }x\in F_{n,m}\setminus N_{n,m}\text{ and }q\in\Q^n,\\
\text{ for every }x\in N_{n,m}\text{ and }q\in\Q^n.
\end{array}\right.
\end{equation}
Property \eqref{eq:Gamma_full_aux} grants that
$\overline\varphi(x,\cdot):\,\big(\Q^n,\nnorm_1(x,\cdot)\big)\to\big(\R^m,\nnorm_2(x,\cdot)\big)$
is a $\Q$-linear $1$-Lipschitz operator for all $x\in F_{n,m}$, whence it can be uniquely
extended to an $\R$-linear $1$-Lipschitz operator $\overline\varphi(x,\cdot):\,
\big(\R^n,\nnorm_1(x,\cdot)\big)\to\big(\R^m,\nnorm_2(x,\cdot)\big)$.
This defines a map $\overline\varphi:\,T_1\to T_2$. To show that such map is an MBB
pre-morphism, it only remains to check its measurability, which amounts to proving that
$\overline\varphi\restr{F_{n,m}\times\R^n}:\,F_{n,m}\times\R^n\to F_{n,m}\times\R^m$
is Borel for every $n,m\in\N$. We actually show that each $\overline\varphi\restr{F_{n,m}\times\R^n}$
is a Carath\'{e}odory map: for any $x\in F_{n,m}$ we have that $\overline\varphi(x,\cdot)$
is continuous by its very construction, while for any vector $c\in\R^n$ we have that the map
$F_{n,m}\ni x\mapsto\overline\varphi(x,c)\in\R^m$ is Borel as pointwise limit of
the Borel maps $\nchi_{F_{n,m}}\,\overline{\Phi(\boldsymbol q^k)}$, where
$(q^k)_{k\in\N}\subseteq\Q^n$ is any sequence converging to $c$.
Hence let us define $\varphi\in{\rm Mor}(\T_1,\T_2)$ as the equivalence class of
the MBB pre-morphism $\overline\varphi$.

We conclude by proving that $\Gamma(\varphi)=\Phi$. For any $n\in\N$ and $q\in\Q^n$,
we have that a representative of $\Gamma(\varphi)(\boldsymbol q)$ is given by the
map $\overline\varphi\circ\overline{\boldsymbol q}$, which $\mm$-a.e.\ coincides
in $E^1_n$ with $\overline{\Phi(\boldsymbol q)}$, whence $\Gamma(\varphi)(\boldsymbol q)=\Phi(\boldsymbol q)$.
Since both $\Gamma(\varphi)$ and $\Phi$ are $L^0(\mm)$-linear, we thus immediately
deduce that $\Gamma(\varphi)(t)=\Phi(t)$ for every $t\in\sfS(\T_1)$.
Finally, the density of $\sfS(\T_1)$ in $\Gamma(\T_1)$, proven in Lemma \ref{lemma:density_simple_sections},
together with the continuity of $\Gamma(\varphi)$ and $\Phi$, grant
that $\Gamma(\varphi)=\Phi$, as required. Therefore the section functor $\Gamma$ is full.
\end{proof}

We now collect the last two results, thus obtaining the main theorem of the paper:
\begin{theorem}[Serre-Swan]\label{thm:Serre-Swan}
Let $\XX=(\X,\sfd,\mm)$ be a metric measure space. Then the section functor
$\Gamma:\,\mathbf{MBB}(\XX)\to\mathbf{NMod}_{\rm pr}(\XX)$ on $\XX$ is an
equivalence of categories.
\end{theorem}
\begin{proof}
By \cite[Proposition 7.25]{awodey2006category} it suffices to prove that
the functor $\Gamma$ is fully faithful and `essentially surjective', the
latter meaning that for each object $\mathscr M$ of $\mathbf{NMod}_{\rm pr}(\XX)$
there exists an object $\T$ of $\mathbf{MBB}(\XX)$ such that $\Gamma(\T)$
and $\mathscr M$ are isomorphic. Therefore Proposition \ref{prop:Gamma_essentially_surj}
and Proposition \ref{prop:Gamma_full_faithful} yield the thesis.
\end{proof}
\section{Some further constructions}
\subsection{Hilbert modules and measurable Hilbert bundles}
An important class of $L^0$-normed $L^0$-modules is that of \emph{Hilbert modules}, defined as follows:
\begin{definition}[Hilbert module]
Let $\mathscr M$ be an $L^0(\mm)$-normed $L^0(\mm)$-module. Then we say that
$\mathscr M$ is a \emph{Hilbert module} provided it satisfies the \emph{pointwise parallelogram rule}, i.e.\
\begin{equation}
|v+w|^2+|v-w|^2=2\,|v|^2+2\,|w|^2\;\;\;\mm\text{-a.e.\ }\quad\text{for every }v,w\in\mathscr M.
\end{equation}
\end{definition}

We shall denote by $\mathbf{HNMod}_{\rm pr}(\XX)$ the subcategory of $\mathbf{NMod}_{\rm pr}(\XX)$
made of those modules that are Hilbert modules. Our goal is to characterise those
measurable Banach bundles that correspond to the Hilbert modules via the section functor $\Gamma$.
As one might expect, such bundles are precisely the ones given by the following definition:
\begin{definition}[Measurable Hilbert bundle]
Let $\T$ be a measurable Banach bundle over the space $\XX$. Then we say that $\T$ is
a \emph{measurable Hilbert bundle}, or briefly \emph{MHB}, provided for one (thus any)
representative $\overline\T=(T,\underline E,\pi,\nnorm)$ of $\T$ it holds that $\nnorm(x,\cdot)$
is a norm induced by a scalar product for $\mm$-a.e.\ point $x\in\X$.
\end{definition}

Given any such point $x\in\X$, we denote the associated scalar product on $(\overline\T)_x$ by
\begin{equation}\label{eq:scalar_product}
\big\langle(x,v),(x,w)\big\rangle_x
:=\frac{\nnorm(x,v+w)^2-\nnorm(x,v)^2-\nnorm(x,w)^2}{2}
\end{equation}
for every $(x,v),(x,w)\in(\overline\T)_x$.

We shall denote by $\mathbf{MHB}(\XX)$ the subcategory of $\mathbf{MBB}(\XX)$ made of those
bundles that are measurable Hilbert bundles. Therefore we can easily prove that:
\begin{proposition}\label{prop:Hilbert_case}
Let $\T$ be a measurable Banach bundle over $\XX$. Then $\T$ is a measurable Hilbert bundle
if and only if $\Gamma(\T)$ is a Hilbert module.
\end{proposition}
\begin{proof}
Choose any representative $\overline\T=(T,\underline E,\pi,\nnorm)$ of
the measurable Banach bundle $\T$.\\
\textsc{Necessity.} Suppose that $\T$ is a measurable Hilbert bundle.
This means that $\nnorm(x,\cdot)$ satisfies the parallelogram rule
for $\mm$-a.e.\ $x\in\X$. Now let $s_1,s_2\in\Gamma(\T)$ be fixed and choose
some representatives $\overline s_1,\overline s_2:\,\X\to T$. Hence for
$\mm$-a.e.\ $x\in\X$ it holds that
\[\begin{split}
|\overline s_1+\overline s_2|^2(x)+|\overline s_1-\overline s_2|^2(x)
&=\big(\nnorm\circ(\overline s_1+\overline s_2)\big)^2(x)
+\big(\nnorm\circ(\overline s_1-\overline s_2)\big)^2(x)\\
&=2\,(\nnorm\circ\overline s_1)^2(x)+2\,(\nnorm\circ\overline s_2)^2(x)\\
&=2\,|\overline s_1|^2(x)+2\,|\overline s_2|^2(x),
\end{split}\]
which grants that $|s_1+s_2|^2+|s_1-s_2|^2=2\,|s_1|^2+2\,|s_2|^2$ holds $\mm$-a.e.\ in $\X$.
Therefore $\Gamma(\T)$ is a Hilbert module by arbitrariness of $s_1,s_2\in\Gamma(\T)$.\\
\textsc{Sufficiency.} Suppose that $\Gamma(\T)$ is Hilbert module. Let $n\in\N$ be fixed.
For any $q\in\Q^n$, consider $\overline{\boldsymbol q}:\,\X\to T$ and
$\boldsymbol q\in\Gamma(\T)$ as in Remark \ref{rmk:constant_sections}.
Then there exists an $\mm$-negligible Borel
subset $N_n$ of $E_n$ such that $\nnorm(x,\cdot)$ is a norm and the equality
\[|\overline{\boldsymbol q}_1+\overline{\boldsymbol q}_2|^2(x)
+|\overline{\boldsymbol q}_1-\overline{\boldsymbol q}_2|^2(x)
=2\,|\overline{\boldsymbol q}_1|^2(x)+
2\,|\overline{\boldsymbol q}_2|^2(x)
\quad\text{ for every }q_1,q_2\in\Q^n\]
is satisfied for every point $x\in E_n\setminus N_n$. This implies that
\[\nnorm(x,q_1+q_2)^2+\nnorm(x,q_1-q_2)^2=2\,\nnorm(x,q_1)^2+2\,\nnorm(x,q_2)^2
\quad\text{ for every }x\in E_n\setminus N_n\text{ and }q_1,q_2\in\Q^n.\]
Therefore $\nnorm(x,\cdot)$ satisfies the parallelogram rule for every $x\in E_n\setminus N_n$
by continuity, so that accordingly $\T$ is a measurable Hilbert bundle.
\end{proof}

As a consequence of Proposition \ref{prop:Hilbert_case}, we finally conclude that
\begin{theorem}[Serre-Swan for Hilbert modules]
The section functor $\Gamma$ restricts to an equivalence of categories
between $\mathbf{MHB}(\XX)$ and $\mathbf{HNMod}_{\rm pr}(\XX)$.
\end{theorem}
\begin{remark}{\rm
It has been proved in \cite[Theorem 1.4.11]{Gigli14} that any separable
Hilbert module (thus in particular any proper Hilbert module by Remark \ref{rmk:prop_mod_separable})
is exactly the space of sections of a suitable measurable Hilbert bundle.
Moreover, as pointed out in \cite[Remark 1.4.12]{Gigli14}, this theory of
Hilbert modules coincides with that of direct integral of Hilbert spaces
(cf.\ \cite{Takesaki79}). More precisely, under this identification, a
Hilbert module corresponds to a measurable field of Hilbert spaces.
\fr}\end{remark}

Let $\mathscr H_1$, $\mathscr H_2$ be two given Hilbert modules over $\XX$.
Then we can consider their tensor product $\mathscr H_1\otimes\mathscr H_2$,
which is a Hilbert module over $\XX$ as well (cf.\ \cite[Section 1.5]{Gigli14}).
\begin{remark}{\rm
Suppose that $\mathscr H_1$ and $\mathscr H_2$ are proper, with dimensional
decomposition $(E^1_n)_{n\in\N}$ and $(E^2_m)_{m\in\N}$, respectively.
Then it can be readily checked that $\mathscr H_1\otimes\mathscr H_2$ has dimension
equal to $nm$ on $E^1_n\cap E^2_m$ for any $n,m\in\N$. In particular, the dimensional
decomposition $(E_k)_{k\in\N}$ of $\mathscr H_1\otimes\mathscr H_2$ is given by
\begin{equation}\label{eq:dd_tensor}
E_k:=\bigcup_{\substack{n,m\in\N: \\ nm=k}}E^1_n\cap E^2_m\quad\text{ for every }k\in\N,
\end{equation}
so that $\mathscr H_1\otimes\mathscr H_2$ is a proper module as well.
\fr}\end{remark}

On the other hand, we now define the tensor product of two MHB's in the following way:
\begin{definition}[Tensor product of MHB's]
Let $\T_1$, $\T_2$ be measurable Hilbert bundles over $\XX$. Choose two representatives
$\overline\T_1$ and $\overline\T_2$, say $\overline\T_i=(T_i,\underline E^i,\pi_i,\nnorm_i)$
for $i=1,2$. Let us define $\underline E=(E_k)_{k\in\N}$ as in \eqref{eq:dd_tensor} and
$T$, $\pi$ accordingly. Given $n,m\in\N$ and $x\in E^1_n\cap E^2_m$ such that
$\nnorm_1(x,\cdot)$, $\nnorm_2(x,\cdot)$ are norms induced by a scalar product, we define
\begin{equation}
\nnorm(x,c):=\bigg(\sum_{j,j'=1}^n\,\sum_{\ell,\ell'=1}^m c_{(j-1)m+\ell}\,c_{(j'-1)m+\ell'}
\,\big\langle(x,\sfe_j),(x,\sfe_{j'})\big\rangle_{1,x}
\,\big\langle(x,\sff_\ell),(x,\sff_{\ell'})\big\rangle_{2,x}\bigg)^{1/2}
\end{equation}
for every $c=(c_1,\ldots,c_{nm})\in\R^{nm}$, where $\sfe_1,\ldots,\sfe_n$ and
$\sff_1,\ldots,\sff_m$ denote the canonical bases of $\R^n$ and $\R^m$, respectively,
while $\la\cdot,\cdot\ra_{i,x}$ stays for the scalar product on $(\overline\T_i)_x$
as in \eqref{eq:scalar_product}.
Then we define the \emph{tensor product} $\T_1\otimes\T_2$ as the equivalence
class of $(T,\underline E,\pi,\nnorm)$,
which turns out to be a measurable Hilbert bundle over $\XX$.
\end{definition}

Given any real number $\lambda\in\R$, we shall write $\lceil\lambda\rceil\in\Z$
to indicate the smallest integer number that is greater than or equal to $\lambda$.
\begin{theorem}
Let $\T_1$, $\T_2$ be measurable Hilbert bundles over $\XX$. Then
\begin{equation}
\Gamma(\T_1)\otimes\Gamma(\T_2)=\Gamma(\T_1\otimes\T_2).
\end{equation}
\end{theorem}
\begin{proof}
We build an operator $\iota:\,\Gamma(\T_1)\otimes\Gamma(\T_2)\to\Gamma(\T_1\otimes\T_2)$
in the following way: first of all, fix $s^1\in\Gamma(\T_1)$ and $s^2\in\Gamma(\T_2)$.
Choose representatives $\overline\T_i=(T_i,\underline E^i,\pi_i,\nnorm_i)$ and
$\overline s^i:\,\X\to T_i$ for $i=1,2$. Given $n,m\in\N$, $x\in E^1_n\cap E^2_m$
and called $\overline s^1(x)=(x,v)$, $\overline s^2(x)=(x,w)$, we define
\[\overline s(x):=(x,c),\quad\text{ where }
c_k:=v_{\lceil k/m\rceil}\,w_{k-m\lceil k/m\rceil+m}
\text{ for all }k=1,\ldots,nm.\]
Hence the equivalence class $\iota(s^1\otimes s^2)$ of $\overline s$
is a section of $\T_1\otimes\T_2$. Simple computations yield
\[\big|\iota(s^1\otimes s^2)\big|=\sqrt{|s^1|\,|s^2|}=|s^1\otimes s^2|
\quad\mm\text{-a.e.\ on }\X.\]
Therefore $\iota$ can be uniquely extended to the whole $\Gamma(\T_1)\otimes\Gamma(\T_2)$
by linearity and continuity, thus obtaining an $L^0(\mm)$-linear operator
$\iota:\,\Gamma(\T_1)\otimes\Gamma(\T_2)\to\Gamma(\T_1\otimes\T_2)$ that preserves
the pointwise norm. In order to conclude, it only remains to check that such $\iota$
is surjective. Fix $n,m\in\N$ and call $(\sfe_i)_{i=1}^n$, $(\sff_j)_{j=1}^m$ and
$(\sfg_k)_{k=1}^{nm}$ the canonical bases of $\R^n$, $\R^m$ and $\R^{nm}$, respectively.
Denote by $\boldsymbol{\sfe}_i\in\Gamma(\T_1)$, $\boldsymbol{\sff}_j\in\Gamma(\T_2)$
and $\boldsymbol{\sfg}_k\in\Gamma(\T_1\otimes\T_2)$ the associated constant sections.
It is then easy to realise that
\[
\nchi_{E^1_n\cap E^2_m}\cdot\boldsymbol{\sfg}_k=
\iota\big((\nchi_{E^1_n\cap E^2_m}\cdot\boldsymbol{\sfe}_{\lceil k/m\rceil})\otimes
(\nchi_{E^1_n\cap E^2_m}\cdot\boldsymbol{\sff}_{k-m\lceil k/m\rceil+m})\big)
\quad\text{ for every }k=1,\ldots,nm.\]
Hence the set $(\nchi_{E^1_n\cap E^2_m}\cdot\boldsymbol{\sfg}_k)_{k=1}^{nm}$,
which forms a local basis for $\Gamma(\T_1\otimes\T_2)$ on $E^1_n\cap E^2_m$, is contained
in the range of the map $\iota$. This grants that $\iota$ is surjective, as required.
\end{proof}
\subsection{Pullbacks and duals}
Let $\XX=(\X,\sfd_\X,\mm_\X)$ and $\YY=({\rm Y},\sfd_{\rm Y},\mm_{\rm Y})$ be
fixed metric measure spaces.

Let $f:\,\X\to{\rm Y}$ be any map \emph{of bounded compression},
namely a Borel map such that
\begin{equation}\label{eq:bdd_comp}
f_*\mm_\X\leq C\,\mm_{\rm Y}\quad\text{ for some constant }C>0.
\end{equation}
Given any $L^0(\mm_{\rm Y})$-normed $L^0(\mm_{\rm Y})$-module $\mathscr M$,
there is a natural way to define its \emph{pullback module} $f^*\mathscr M$,
which is an $L^0(\mm_\X)$-normed $L^0(\mm_\X)$-module (see \cite{GP17}
or \cite{GR17} for the details).
\begin{remark}{\rm
Since we are dealing with normed modules over $L^0$, so that there is no integrability
requirement on the pointwise norm of the elements of our modules, it is natural
to ask whether the hypothesis \eqref{eq:bdd_comp} on $f$ can be weakened.
Actually, this is the case: one can build the pullback module under the only
assumption that $f_*\mm_\X\ll\mm_{\rm Y}$. Indeed, choose any Borel representative
$A$ of $\big\{\frac{\d f_*\mm_\X}{\d\mm_{\rm Y}}>0\big\}$, where $\frac{\d f_*\mm_\X}{\d\mm_{\rm Y}}$
denotes the Radon-Nikod\'{y}m derivative of $f_*\mm_\X$ with respect to $\mm_{\rm Y}$.
Then define $\mm'_{\rm Y}:=(f_*\mm_\X)\restr A+\mm_{\rm Y}\restr{\X\setminus A}$.
Hence $\mm_{\rm Y}\ll\mm'_{\rm Y}\ll\mm_{\rm Y}$, which grants that
$L^0(\mm'_{\rm Y})=L^0(\mm_{\rm Y})$ and accordingly that $\mathscr M$ is
an $L^0(\mm'_{\rm Y})$-normed $L^0(\mm'_{\rm Y})$-module.
Moreover, it holds that $f_*\mm_\X\leq\mm'_{\rm Y}$, which says that $f$ is a map
of bounded compression when the target $\rm Y$ is endowed with the measure $\mm'_{\rm Y}$,
so that it makes sense to consider $f^*\mathscr M$.
\fr}\end{remark}

We now define what is the pullback of a measurable Banach bundle over $\YY$:
\begin{definition}[Pullback of an MBB]
Let $\T$ be a measurable Banach bundle over $\YY$. Choose a representative
$\overline\T=(T,\underline E,\pi,\nnorm)$ of $\T$.
Let us set
\begin{equation}\label{eq:def_pullback}\begin{split}
\underline E'&:=\big(f^{-1}(E_n)\big)_{n\in\N}\\
\nnorm'(x,v)&:=\nnorm\big(f(x),v\big)
\end{split}\quad\begin{split}
&\text{ and }T',\pi'\text{ accordingly,}\\
&\text{ for every }(x,v)\in T'.
\end{split}\end{equation}
Then we define the \emph{pullback bundle}
$f^*\T$ as the equivalence class of $(T',\underline E',\pi',\nnorm')$,
which turns out to be a measurable Banach bundle over $\XX$.
\end{definition}
\begin{theorem}
Let $\T$ be a measurable Banach bundle over $\YY$. Then
\begin{equation}
f^*\Gamma(\T)=\Gamma(f^*\T).
\end{equation}
\end{theorem}
\begin{proof}
We aim to build a linear map $f^*:\,\Gamma(\T)\to\Gamma(f^*\T)$ such that
\begin{equation}\label{eq:pullbacks_aux}\begin{split}
|f^*s|=|s|\circ f&\quad\text{ for every }s\in\Gamma(\T),\\
\big\{f^*s\,:\,s\in\Gamma(\T)\big\}&\quad\text{ generates }\Gamma(f^*\T).
\end{split}\end{equation}
Pick a representative $\overline\T=(T,\underline E,\pi,\nnorm)$ of $\T$ and
define $(T',\underline E',\pi',\nnorm')$ as in \eqref{eq:def_pullback}.
Take $s\in\Gamma(\T)$, with representative $\overline s:\,{\rm Y}\to T$.
Given any $x\in\X$, let us define $\overline s'(x):=(x,v)$, where $v$ is the
unique vector for which $\overline s\big(f(x)\big)=\big(f(x),v\big)$.
It clearly holds that $\overline s':\,\X\to T'$ is a section of
the MBB $(T',\underline E',\pi',\nnorm')$.
Then we define $f^*s$ as the equivalence class of $\overline s'$.
We thus built a map $f^*:\,\Gamma(\T)\to\Gamma(f^*\T)$,
which is linear and satisfies the first in \eqref{eq:pullbacks_aux} by construction.

Now fix $n\in\N$ and $q\in\Q^n$. Denote by $\boldsymbol q\in\Gamma(\T)$ and
$\boldsymbol q'\in\Gamma(f^*\T)$ the constant sections associated to $q$.
It is then easy to check that $\boldsymbol q'=f^*\boldsymbol q$. This grants that
\[\sfS(f^*\T)\subseteq\left\{\sum_{i\in\N}\nchi_{A_i}\cdot f^*s_i
\;\bigg|\;(A_i)_i\text{ is a Borel partition of }\X\text{, }(s_i)_i\subseteq\Gamma(\T)\right\}.\]
Since $\sfS(f^*\T)$ is dense in $\Gamma(f^*\T)$ by Lemma \ref{lemma:density_simple_sections},
we finally conclude that the second in \eqref{eq:pullbacks_aux} is verified as well.
Therefore the statement is achieved.
\end{proof}
\bigskip

We now introduce the notions of dual module and of dual bundle.\\
Let $\mathscr M$ be an $L^0(\mm)$-normed $L^0(\mm)$-module. We define
the \emph{dual module} $\mathscr M^*$ as the space of all
$L^0(\mm)$-linear continuous maps $T:\,\mathscr M\to L^0(\mm)$,
endowed with the pointwise norm
\begin{equation}
|T|_*:=\underset{\substack{v\in\mathscr M: \\ |v|\leq 1}}{\rm ess\,sup\,}
\big|T(v)\big|\in L^0(\mm)\quad\text{ for every }T\in\mathscr M^*.
\end{equation}
It can be readily proven that $\mathscr M^*$ has a natural structure of $L^0(\mm)$-normed $L^0(\mm)$-module.
\begin{definition}[Dual bundle]
Let $\T$ be a measurable Banach bundle over $\XX$. Choose a representative
$\overline\T=(T,\underline E,\pi,\nnorm)$ of $\T$. Let us set
\begin{equation}
\nnorm^*(x,v):=\left\{\begin{array}{ll}
\sup_{w\in(\overline\T)_x\setminus\{0\}}\frac{|v\cdot w|}{\nnorm(x,w)}\\
0
\end{array}\quad\begin{array}{ll}
\text{ if }\nnorm(x,\cdot)\text{ is a norm,}\\
\text{ otherwise.}
\end{array}\right.
\end{equation}
Then we define the \emph{dual bundle} $\T^*$ as the equivalence class of
$(T,\underline E,\pi,\nnorm^*)$, which turns out to be a measurable Banach
bundle over $\XX$.
\end{definition}
\begin{theorem}
Let $\T$ be a measurable Banach bundle over $\XX$. Then
\begin{equation}
\Gamma(\T)^*=\Gamma(\T^*).
\end{equation}
\end{theorem}
\begin{proof}
Consider the operator $\iota:\,\Gamma(\T^*)\to\Gamma(\T)^*$ defined as follows:
given any $s^*\in\Gamma(\T^*)$, we call $\iota(s^*):\,\Gamma(\T)\to L^0(\mm)$ the map
sending (the equivalence class of) any section $\overline s$ to the function $\X\ni x\mapsto\overline s^*(x)\cdot\overline s(x)\in\R$,
where $\overline s^*$ is any representative of $s^*$. One can easily deduce
from its very construction that $\iota$ is a module morphism that
preserves the pointwise norm. To conclude, it only remains to show that the map $\iota$ is surjective.
Let $T\in\Gamma(\T)^*$ be fixed. For any $n\in\N$, denote by $\sfe^n_1,\ldots,\sfe^n_n$
the canonical basis of $\R^n$ and by $\boldsymbol{\sfe}^n_1,\ldots,\boldsymbol{\sfe}^n_n\in\Gamma(\T)$ the associated constant sections. Hence let us define
$s^*\in\Gamma(\T^*)$ as
\[s^*(x):=\sum_{n\in\N}
\Big(x,\big(T\boldsymbol{\sfe}^n_1(x),\ldots,T\boldsymbol{\sfe}^n_n(x)\big)\Big)
\quad\text{ for }\mm\text{-a.e.\ }x\in\X.\]
Simple computations show that $\iota(s^*)=T$. Hence $\iota$ is surjective,
concluding the proof.
\end{proof}
\appendix
\section{Comparison with the Serre-Swan theorem for smooth manifolds}\label{ap:comparison}
We point out the main analogies and differences between our work and
the Serre-Swan theorem for smooth manifolds, for whose
presentation we refer to \cite[Chapter 11]{nestruev2003smooth}.

The result in the smooth case can be informally stated as follows:
\emph{the category of smooth vector bundles over a connected manifold
$M$ is equivalent to the category of finitely-generated projective $C^\infty(M)$-modules.}

In our non-smooth setting we had to replace `smooth' with `measurable',
in a sense, and this led to these discrepancies with the case of manifolds: 
\begin{itemize}
\item[$\rm i)$] The fibers of a measurable Banach bundle need not have
the same dimension (still, they are finite dimensional), while on
a connected manifold any smooth vector bundle must have constant
dimension by topological reasons.
\item[$\rm ii)$] In the definition of measurable Banach bundle
we do not speak about the analogous of the `trivialising diffeomorphisms',
the reason being that one can always patch together countably many measurable
maps still obtaining a measurable map. Hence there is no loss of generality in requiring
the total space of the bundle to be of the form $\bigsqcup_{n\in\N}E_n\times\R^n$ and
its measurable subsets to be those sets whose intersection with each $E_n\times\R^n$ is a Borel set.
\item[$\rm iii)$] Given that we want to correlate the measurable Banach bundles
with the $L^0(\mm)$-normed $L^0(\mm)$-modules, which are naturally equipped with
a pointwise norm $|\cdot|$, we also require the existence of a function $\nnorm$
that assigns a norm to (almost) every fiber of our bundle. A similar structure is
not treated in the smooth case.
\item[$\rm iv)$]
The Serre-Swan theorem for smooth manifolds deals with modules that are
finitely-generated and projective. In our context, any finitely-generated module
is automatically projective, as seen in Proposition \ref{prop:fin-gen_are_proj}.
Moreover, the flexibility of $L^0(\mm)$ actually allowed us to extend the result to all proper modules,
that are not necessarily `globally' finitely-generated but only
`locally' finitely-generated, in a sense.
\end{itemize}
\section{A variant for \texorpdfstring{$L^p$}{Lp}-normed \texorpdfstring{$L^\infty$}{Linfty}-modules}\label{ap:variant}
The original presentation of the concept of $L^0$-normed $L^0$-module, which has
been proposed in \cite{Gigli14}, follows a different line of thought with respect
to the one presented here. In \cite{Gigli14} it is first given the notion
of \emph{$L^p$-normed $L^\infty$-module}, then by suitably completing such objects
one obtains the class of $L^0$-normed $L^0$-modules. The role of this completion
is to `remove any integrability requirement'. On the other hand, the axiomatisation
of $L^0$-normed $L^0$-modules that we presented in Subsection \ref{subsec:L0_norm_L0_mod}
is taken from \cite{Gigli17}.

Our choice of using the language of $L^0$-normed $L^0$-modules, instead of
$L^p$-normed $L^\infty$-modules, is only a matter of practicity and
is not due to any theoretical reason.
Indeed, in this appendix we show that all the results we obtained so far
can be suitably reformulated for $L^p$-normed $L^\infty$-modules.
\bigskip

Let $\XX=(\X,\sfd,\mm)$ be a given metric measure space. Fix an exponent $p\in[1,\infty]$.
In order to keep a distinguished notation, we shall indicate by $\mathscr M^p$ the
$L^p(\mm)$-normed $L^\infty(\mm)$-modules, for whose definition and properties we
refer to \cite{Gigli14} or \cite{Gigli17}, while the $L^0(\mm)$-normed $L^0(\mm)$-modules
will be denoted by $\mathscr M^0$. The category of $L^p(\mm)$-normed $L^\infty(\mm)$-modules
is denoted by $\mathbf{NMod}^p(\XX)$ and that of $L^0(\mm)$-normed $L^0(\mm)$-modules
by $\mathbf{NMod}^0(\XX)$. Moreover, the subcategories of $\mathbf{NMod}^p(\XX)$
and $\mathbf{NMod}^0(\XX)$ that consist of all proper modules will be called
$\mathbf{NMod}^p_{\rm pr}(\XX)$ and $\mathbf{NMod}^0_{\rm pr}(\XX)$, respectively.
Observe that we added the exponent $0$ to the notation of Definition \ref{def:NMod_pr}.
Similarly, we shall denote by $\Gamma_0$ the section functor $\Gamma$ that has been
introduced in Definition \ref{def:section_functor}.

The following results look upon the relation that subsists between the class of
$L^p(\mm)$-normed $L^\infty(\mm)$-modules and that of $L^0(\mm)$-normed $L^0(\mm)$-modules.
First of all, it has been proved in \cite[Theorem/Definition 1.7]{Gigli17} that
\begin{theorem}[$L^0$-completion]
Let $\mathscr M^p$ be an $L^p(\mm)$-normed $L^\infty(\mm)$-module. Then there
exists a unique couple $(\mathscr M^0,\iota)$, called \emph{$L^0$-completion} of $\mathscr M^p$,
where $\mathscr M^0$ is an $L^0(\mm)$-normed $L^0(\mm)$-module and $\iota:\,\mathscr M^p\to\mathscr M^0$
is a linear map with dense image that preserves the pointwise norm. Uniqueness has to
be intended up to unique isomorphism.
\end{theorem}

It can be easily seen that the local dimension of a module is invariant under
taking the $L^0$-completion, namely for any Borel set $E\subseteq\X$
with $\mm(E)>0$ and for any $n\in\N$ it holds
\begin{equation}\label{eq:invariant_dim_under_completion}
\mathscr M^p\text{ has dimension }n\text{ on }E
\quad\Longleftrightarrow\quad
\mathscr M^0\text{ has dimension }n\text{ on }E.
\end{equation}
Given two $L^p(\mm)$-normed $L^\infty(\mm)$-modules $\mathscr M^p$, $\mathscr N^p$ and
a module morphism $\Phi:\,\mathscr M^p\to\mathscr N^p$, there exists a unique module
morphism $\widetilde\Phi:\,\mathscr M^0\to\mathscr N^0$ extending $\Phi$, where
$\mathscr M^0$ and $\mathscr N^0$ denote the $L^0$-completions of $\mathscr M^p$
and $\mathscr N^p$, respectively.
\begin{definition}[$L^0$-completion functor]
We define the \emph{$L^0$-completion functor} as the functor $\sfC^p:\,\mathbf{NMod}^p(\XX)\to\mathbf{NMod}^0(\XX)$
that assigns to any $\mathscr M^p$ its $L^0$-completion $\mathscr M^0$
and to any module morphism $\Phi:\,\mathscr M^p\to\mathscr N^p$ its
unique extension $\widetilde\Phi:\,\mathscr M^0\to\mathscr N^0$.
\end{definition}

Conversely, given any $L^0(\mm)$-normed $L^0(\mm)$-module $\mathscr M^0$, one has that
\begin{equation}\label{eq:restr_L0_mod}
\mathscr M^p:=\big\{v\in\mathscr M^0\;\big|\;|v|\in L^p(\mm)\big\}
\quad\text{ has a structure of }L^p(\mm)\text{-normed }L^\infty(\mm)\text{-module.}
\end{equation}
Moreover, it holds that the $L^0$-completion of $\mathscr M^p$ is the original module $\mathscr M^0$.
\begin{definition}[$L^p$-restriction functor]
The \emph{$L^p$-restriction functor} is defined as that functor
$\sfR^p:\,\mathbf{NMod}^0(\XX)\to\mathbf{NMod}^p(\XX)$
that assigns to any $\mathscr M^0$ its `restriction' $\mathscr M^p$, as in \eqref{eq:restr_L0_mod},
and to any module morphism $\widetilde\Phi:\,\mathscr M^0\to\mathscr N^0$ its
restriction $\Phi:=\widetilde\Phi\restr{\mathscr M^p}:\,\mathscr M^p\to\mathscr N^p$,
which turns out to be a morphism of $L^p(\mm)$-normed $L^\infty(\mm)$-modules.
\end{definition}

We can finally collect all of the properties described so far in the following statement:
\begin{theorem}[$\mathbf{NMod}^p(\XX)$ is equivalent to $\mathbf{NMod}^0(\XX)$]
Both the functors $\sfC^p$ and $\sfR^p$ are equivalence of categories, one the inverse of the other.
\end{theorem}

Property \eqref{eq:invariant_dim_under_completion} ensures that $\mathscr M^p$
and $\sfC^p(\mathscr M^p)$ have the same dimensional decomposition, thus
in particular the above functors naturally restrict to
$\sfC^p_{\rm pr}:\,\mathbf{NMod}^p_{\rm pr}(\XX)\to\mathbf{NMod}^0_{\rm pr}(\XX)$
and $\sfR^p_{\rm pr}:\,\mathbf{NMod}^0_{\rm pr}(\XX)\to\mathbf{NMod}^p_{\rm pr}(\XX)$. Therefore:
\begin{corollary}[$\mathbf{NMod}^p_{\rm pr}(\XX)$ is equivalent to $\mathbf{NMod}^0_{\rm pr}(\XX)$]\label{cor:equiv_Lp_L0}
The functors $\sfC^p_{\rm pr}$ and $\sfR^p_{\rm pr}$ are equivalence of categories, one the inverse of the other.
\end{corollary}

Now fix a measurable Banach bundle $\T$ over $\XX$. Then let us define
\begin{equation}\label{eq:def_Gamma_p(T)}
\Gamma_p(\T):=\big\{s\in\Gamma_0(\T)\;\big|\;|s|\in L^p(\mm)\big\}.
\end{equation}
The space $\Gamma_p(\T)$ can be viewed as an $L^p(\mm)$-normed $L^\infty(\mm)$-module.
Moreover, given any two measurable Banach bundles $\T_1$, $\T_2$ over $\XX$ and a bundle morphism
$\varphi\in{\rm Mor}(\T_1,\T_2)$, let us define
$\Gamma_p(\varphi)\in{\rm Mor}\big(\Gamma_p(\T_1),\Gamma_p(\T_2)\big)$ as
\begin{equation}
\Gamma_p(\varphi):=\Gamma_0(\varphi)\restr{\Gamma_p(\T_1)}:\,\Gamma_p(\T_1)\to\Gamma_p(\T_2).
\end{equation}
Hence such construction induces an $L^p$-section functor
$\Gamma_p:\,\mathbf{MBB}(\XX)\to\mathbf{NMod}^p_{\rm pr}(\XX)$. Then
\begin{equation}\label{eq:diagram_MBB_NModp_NMod0}\begin{tikzcd}
\mathbf{MBB}(\XX) \arrow{r}{\Gamma_0} \arrow[swap]{rd}{\Gamma_p} &
\mathbf{NMod}^0_{\rm pr}(\XX) \arrow{d}{\sfR^p_{\rm pr}} \\
 & \mathbf{NMod}^p_{\rm pr}(\XX)
\end{tikzcd}\end{equation}
is a commutative diagram. We can finally conclude that
\begin{theorem}[Serre-Swan for $L^p$-normed $L^\infty$-modules]\label{thm:Serre-Swan_Lp_Linfty}
It holds that the $L^p$-section functor $\Gamma_p:\,\mathbf{MBB}(\XX)\to\mathbf{NMod}^p_{\rm pr}(\XX)$ on $\XX$
is an equivalence of categories.
\end{theorem}
\begin{proof}
It follows from Theorem \ref{thm:Serre-Swan}, from Corollary \ref{cor:equiv_Lp_L0}
and from the fact that the diagram in \eqref{eq:diagram_MBB_NModp_NMod0} commutes.
\end{proof}
\bibliographystyle{siam}
{\small
\def\cprime{$'$} \def\cprime{$'$}

}

\end{document}